\documentclass[10pt,letterpaper,twoside]{amsart}
\usepackage{charter, amsmath,amsbsy,amsfonts,amssymb,
stmaryrd,amsthm,mathrsfs,graphicx, amscd, tikz-cd, cancel}
\usepackage{srcltx}

\usepackage[OT2,OT1]{fontenc} \newcommand\cyr
{
\renewcommand\rmdefault{wncyr} \renewcommand\sfdefault{wncyss} \renewcommand\encodingdefault{OT2} \normalfont
\selectfont
}
\DeclareTextFontCommand{\textcyr}{\cyr}    
\makeatletter
   \def\@settitle
{
\begin{center} \baselineskip14\p@\relax \LARGE\@title \end{center}
}
  \makeatother

\usepackage[normalem]{ulem}
\usetikzlibrary{matrix,arrows,decorations.pathmorphing}
\usepackage[
	hypertexnames=false,
	hyperindex,
	pagebackref,
	pdftex,
	breaklinks=true,
	bookmarks=false,
	colorlinks,
	linkcolor=blue,
	citecolor=red,
	urlcolor=red,
]{hyperref}
\usepackage[mathscr]{eucal}
\numberwithin{equation}{section}

\newcommand\bs{\backslash}

\usepackage{verbatim}
\usepackage{amssymb}
\usepackage{amsbsy}
\usepackage{amscd}
\usepackage{amsmath}
\usepackage{amsthm}
\usepackage[mathscr]{eucal}

\def\AA{{\mathbb A}}
\def\Ascr{{\mathscr A}}

\def\CC{{\mathbb C}}

\def\FF{{\mathbb F}} 
 
\def\HH{{\mathbb H}}

\def\QQ{{\mathbb Q}} 
\def\RR{{\mathbb R}} 
 
\def\VV{{\mathbb V}} 
\def\XX{{\mathbb X}} 
\def\ZZ{{\mathbb Z}}

\def\pfrak{{\mathfrak{p}}}

\def\ufrak{{\mathfrak{u}}}
\def\vfrak{{\mathfrak{v}}}

\def\gfrak{\mathfrak{g}}

\def\G{\Gamma}
\def\g{\gamma}

\def\bb{\mathrm{bb}} 
\def\GP{\mathrm{gp}} 
 
\def\cB{\cha^{\text{\hskip -3mm{\cyr B}}}}
 
\def\prim{\mathrm{pr}} 
 
\def\ssm{\smallsetminus}

\def\hor{{\mathit h}}
\def\ver{{\mathit \ell}}

\def\bs{\backslash}

\def\Acal{{\mathcal A}}

\def\Dcal{{\mathcal D}}
\def\Ecal{{\mathcal E}} 
\def\Fcal{{\mathcal F}} 
\def\Gcal{{\mathcal G}} 
 
\def\Iscr{{\mathscr I}}

\def\Ocal{{\mathcal O}}

\def\Pscr{{\mathscr P}}

\def\Scal{{\mathcal S}}
\def\Sscr{{\mathscr S}}

\def\la{\langle}
\def\ra{\rangle}

\def\End{\operatorname{End}}

\def\pt{{\scriptscriptstyle\bullet}}

\newcommand\aty{\operatorname{At}}

\newcommand\cha{\operatorname{ch}}
\newcommand\Ch{\operatorname{Ch}}
\newcommand\chern{\operatorname{c}}
\newcommand\Chern{\operatorname{C}}

\newcommand\first{\operatorname{in}}

\newcommand\Gr{\operatorname{Gr}}
\newcommand\Hom{\operatorname{Hom}}

\newcommand\sym{\operatorname{Sym}}

\newcommand\GL{\operatorname{GL}}
\newcommand\SL{\operatorname{SL}}

\newcommand\Un{\operatorname{U}}

\newcommand\Sp{\operatorname{Sp}}
\newcommand\gr{\operatorname{gr}}

\newcommand\Tr{\operatorname{Tr}}

\newtheorem{theorem}{Theorem}[section]

\newtheorem{lemma}[theorem]{Lemma}

\newtheorem{proposition}[theorem]{Proposition}
\newtheorem{corollary}[theorem]{Corollary}
\newtheorem{cordef}[theorem]{Corollary-definition}

\theoremstyle{definition}
\newtheorem{scholium}[theorem]{Scholium}
\newtheorem{definition}[theorem]{Definition}

\newtheorem{example}[theorem]{Example}

\theoremstyle{remark} 
\newtheorem{remark}[theorem]{Remark}

\newtheorem{cremarks}[theorem]{Concluding remarks}

\begin{document}

\author{Eduard Looijenga}
\address{Yau Mathematical Sciences Center, Tsinghua University Beijing (China) and Mathematisch Instituut, Universiteit Utrecht (Nederland)}
\email{eduard@math.tsinghua.edu.cn}

\def\subjclassname{\textup{2010} Mathematics Subject Classification}
\subjclass[2010]{14G35, 14F43, 32S35}

\title[Chern class lifts  and Tate extensions]{Goresky-Pardon lifts of Chern classes and  associated Tate extensions}

\begin{abstract}
Let $X$ be an irreducible complex-analytic variety, $\Scal$ a stratification of $X$ and $\Fcal$ a holomorphic vector bundle on  the open stratum 
$\mathring{X}$. We give geometric  conditions on $\Scal$ and $\Fcal$ that produce a natural lift of the Chern class $\chern_k(\Fcal)\in H^{2k}(\mathring{X}; \CC)$ to $H^{2k}(X; \CC)$, which, in the algebraic setting, is of Hodge level $\ge k$. When applied to the Baily-Borel compactification $X$ of a locally symmetric variety $\mathring{X}$ and  an automorphic vector bundle $\Fcal$ on $\mathring{X}$, this refines a theorem of Goresky-Pardon. In passing we define a class of simplicial resolutions of the Baily-Borel compactification that can be used to define its mixed Hodge structure. We use this to show that the stable cohomology of the Satake (=Baily-Borel) compactification of $\Acal_g$ contains nontrivial Tate extensions.
\end{abstract}

\maketitle

\section{Introduction}
Let $X$ be an irreducible  complex-analytic variety,  $\mathring{X}$ a nonsingular Zariski open-dense subset of $X$ and $\Fcal$ a holomorphic vector bundle on $\mathring{X}$. In this paper we  give conditions under which the rational Chern classes $\chern_k(\Fcal)\in H^{2k}(X; \QQ)$ extend in a canonical manner as \emph{complex} classes to $X$, even (and especially) in situations where $\Fcal$ is known \emph{not} to extend to $X$ as a complex vector bundle. The passage to complex cohomology is not just an artefact of our method, for  we find examples for which the imaginary part of such a extension is nonzero. 
Before we say more about what is in this paper, we mention the situation that is both the origin and the motivation for addressing this question. This is when $\mathring{X}$ is a locally symmetric variety, $X=(\mathring{X})^\bb$ its Baily-Borel compactification, and $\Fcal$ an automorphic vector bundle on $\mathring{X}$.
Mumford  \cite{mumford:hirz}  defined in 1977 Chern numbers for an automorphic bundle $\Fcal$ as  integrals of Chern forms relative to some metric on $\mathring{X}$ (using his toroidal compactifications to prove their absolute convergence) and proved them to have properties that  Hirzebruch had earlier established  in case $\mathring{X}$ is compact. A quarter of a century later Goresky and Pardon \cite{gp} proved that the Chern classes of such an $\Fcal$ can be naturally extended to $X$ in such a manner that the associated Chern  numbers (i.e., polynomials in these classes evaluated on the fundamental class of $X$) 
yield those of Mumford. 
  
Returning to the content of this article, it has four principal results. The first one may be characterized as putting the result of Goresky and Pardon in 
(what we feel is) its natural setting. This has in any case the effect of making statements more transparent and proofs shorter.  
Key to this approach are the rather simple concepts  formulated in Section \ref{sect:2}. Our point of departure is not just $X$ with its 
Zariski open-dense subset, but rather an analytic  stratification $\Scal$ of $X$ for which $\mathring{X}$ is the union of the open strata. 
We introduce (in  \ref{def:basic2}) certain analytic control data  on $(X, \Scal)$ embodied in the notion of a \emph{system of  local retractions}. 
For a stratification $(X, \Scal)$ thus endowed, we define (in \ref{def:basic3}) a corresponding notion for a holomorphic vector bundle on  
$\mathring{X}$, namely that of an \emph{isoholonomic flat structure}. This structure may also be regarded as a set of control data 
(in the sense  of stratification theory), but now on the vector bundle and compatible with the local retractions. Both notions are analytic in character 
and have algebraic counterparts. Proposition  \ref{prop:main1} states that this last structure suffices to produce a natural lift to $X$ of  the complex Chern classes. We then show that such 
structures are present on the Baily-Borel stratification resp.\ an automorphic vector bundle, so that this recovers the result of Goresky-Pardon.  
We work this out in the case of the symplectic group. 

The second main result pertains to the complete, complex-algebraic setting, where we prove (Theorem \ref{thm:main}) that these Chern class lifts have the 
expected Hodge level, provided that $(X, \Scal)$ admits (what we have called) a \emph{stratified resolution} 
(Definition \ref{def:stratares}). This leads to a simplicial resolution of $X$ by complete nonsingular varieties which satisfies cohomological 
descent so that it can be used to describe the mixed Hodge structure on the cohomology of $X$. 

The third part applies this  to Baily-Borel compactifications:  Theorem \ref{thm:bbres} states that some of Mumford's toroidal resolutions of a Baily-Borel compactification give rise to a stratified resolution. Since these can be used to identify the mixed Hodge  structure on the 
cohomology of a Baily-Borel compactification, we hope that this will find other applications as well.

Our fourth contribution is an application of the preceding to the stable cohomology of the Satake (=Baily-Borel) compactification $\Acal_g^\bb$ of $\Acal_g$. Charney and Lee \cite{charney-lee} have 
shown that for a fixed $k$, $H^k(\Acal_g^\bb; \QQ)$ stabilizes as $g\to \infty$ and that the resulting stable cohomology $H^\pt$ has the structure of a 
$\QQ$-Hopf algebra. This is in fact a polynomial algebra with primitive basis $\widetilde{ch}_{2r+1}\in H^{4r+2}$ ($r\ge 0$) and $y_r\in  H^{4r+2}$ ($r>0$), although the $\widetilde{ch}_{2r+1}$ is not canonically defined (it is a lift of the corresponding Chern character of the Hodge bundle  on $\Acal_g$) and $y_r$ is only defined up to sign. So
$H^{4r+2}_\prim$ is of dimension 2  when $r>0$. Jiaming Chen and the  author \cite{chen-looijenga} have recently shown that $H^\pt$ has a natural mixed Hodge structure that gives $H^{4r+2}_\prim$ 
the structure of a Tate extension: it is an extension of $\QQ(-2r-1)$ (which has the image of $\widetilde{ch}_{2r+1}$ as generator) by $\QQ(0)$ (which has  $y_r$ as generator).
With the help of the results described above, we find that the one-dimensional space $F^{2r+1}H^{4r+2}_\prim$ is in fact spanned by the Goresky-Pardon Chern character of the Hodge bundle on $\Acal_g$ ($g\gg r$). We then use the theory that underlies the construction of the Beilinson regulator to compute the  class of this Tate extension (Theorem \ref{thm:main3}) and find it to be nonzero. At the same time we show that the Goresky-Pardon  Chern character in question  has a real part that is rational (so lies in $H^{4r+2}_\prim$), but that its imaginary part is nonzero. This answers (negatively) the question  asked by Goresky-Pardon ((1.6) of \cite{gp}) whether their lift always lives in rational  cohomology.  Our examples leave open the possibility  that this is so for the real part of this class (say, in the setting of an automorphic vector bundle over the interior of a Baily-Borel compactification).
\\

We close this introduction with a brief discussion of how this is connected with other work in this area.  Goresky and Tai proved in \cite{gt} that an 
automorphic vector bundle on a locally symmetric variety $\mathring{X}$ extends naturally to what is called the \emph{reductive Borel-Serre compactification} of $\mathring{X}$. This compactification, which we shall denote  for the purpose of this introduction by $\widehat X$, has a real-analytic structure and dominates $X$ in the sense that the latter is naturally a quotient of $\widehat X$, but  lives by no means in the complex-analytic category. Goresky and Tai predicted that the  Chern classes of their extension are simply pull-backs of the Goresky-Pardon Chern classes to $\widehat X$ and this was later proved by Zucker \cite{zucker:rbs} (with some corrections supplied by Ayoub and Zucker \cite{az}, see also \cite{nair:ms}). Shortly afterwards  Zucker \cite{zucker:rbs3} showed that  the quotient map $\widehat X\to X$, despite not being in any sense a morphism of complex-algebraic varieties, behaves from a cohomological point of  view as if it were, for he proved that $H^\pt (\widehat  X)$ carries a natural mixed Hodge structure such that the induced map $H^\pt(X)\to H^\pt(\widehat X)$ is a morphism in this category. Very recently Arvind Nair showed in \cite{nair:rbs}  that the Chern classes of the Goresky-Tai extension have the expected Hodge level and he there formulated our second  main result  as a conjecture (a conjectural picture is formulated in subsection (4.3) of \cite{nair:rbs}), something we had not been aware of while working  on this project. In light of Zucker's result, our theorem implies the property proved by Nair,  but is not equivalent to it, as the map $H^\pt(X)\to H^\pt(\widehat X)$ may not be injective. 

In correspondence with Klaus Hulek and others in connection with \cite{chen-looijenga} we had wondered about the possibly nontrivial nature of the above Tate extension. Via audience feedback to a talk of his at the IAS (that apparently had made mention of this question), we learned that the work of Nair might shed light on this and indeed, when we wrote Nair, he informed us (in April 2015) that his techniques---which involve among other things local Hecke operators and analytic results due to Franke---enable  him to determine the class of this extension (which turned out to be nonzero). The proof given here was found thereafter (September 2015), but is,  we understand,  quite different from his.
\\

It is a pleasure to acknowledge the numerous conversations with Spencer Bloch on this material. He drew my attention to the Chern class extensions  defined by Goresky-Pardon, and also suggested  (at a time when neither of us was aware that Nair had in fact conjectured this) that these classes might have the Hodge level property that is established here.  I am also indebted to Mark Goresky, who pointed out to me a subtlety regarding the partial flat structures on automorphic bundles that I had overlooked.

I am also grateful to the comments of two referees.

\section{Chern classes in a stratified setting}\label{sect:2}

\subsection*{Isoholonomic relative connections} Let $\rho: M\to S$ be a submersion of  complex manifolds and let $\Fcal$ be a holomorphic vector bundle on $M$ of rank $r$.  We need the following three notions relative to $\rho$. 

\begin{definition}\label{def:basic1}
 We say that a $C^\infty$-differential form on $M$ is \emph{$\rho$-basic} if  it is locally the pull-back along $\rho$ of a form on $S$. 

A \emph{$\rho$-connection} on $\Fcal$ is a holomorphic connection along the fibers of $\rho$, i.e., is given by a $\rho^{-1}\Ocal_S$-linear map $\nabla_\rho: \Fcal\to \Omega_\rho\otimes\Fcal$ satisfying the Leibniz property:
$\nabla_\rho (\phi s)=\phi \nabla_\rho(s)+d_\rho(\phi)\otimes s$. We say that  it is \emph{flat} if its  curvature form (an $\Ocal_M$-homomorphism $\Fcal\to \Omega^2_\rho\otimes\Fcal$) is identically zero. 
We say that such a flat $\rho$-connection  on $\Fcal$ is  \emph{isoholonomic} if  we can cover $S$ by  open subsets $V$ such that $\nabla_\rho|\rho^{-1}V$ can be lifted to a flat holomorphic connection on $\Fcal |\rho^{-1}V$.
\end{definition}

Let us comment on these definitions. We begin with  observing  that if $\rho$ factors through a submersion 
$\rho': M\to S'$ and one of the three properties above holds for $\rho$, then that property also holds for $\rho'$.

Next we note that  we can drop the adjective `locally' in the definition of a $\rho$-basic form if the fibers of $\rho$ are connected: it is then just the pull-back of a form on $S$. This is still true  if the set of  $v\in S$ for which $\rho^{-1}(v)$ is connected contains an open-dense subset of $S$
(we then say that \emph{a general fiber of $\rho$ is connected}).

For a  flat $\rho$-connection $\nabla_\rho$ on $\Fcal$, its flat  local sections make up a subsheaf $\FF\subseteq\Fcal$. This is a locally free $\rho^{-1}\Ocal_S$-submodule of rank $r$ with the property that the natural map
$\Ocal_M\otimes_{\rho^{-1}\Ocal_S}\FF\to \Fcal$ is an isomorphism. The converse also holds: a subsheaf $\FF\subseteq\Fcal$
with these properties determines  a flat $\rho$-connection on $\Fcal$. This also amounts to giving a (maximal) atlas  of local holomorphic trivializations  of $\Fcal$ whose transition functions factor through $\rho$.
In the  situations that we shall consider, 
$\rho$ will be  topologically locally trivial with connected fibers, and then a given flat $\rho$-connection is isoholonomic precisely if its holonomy along $\rho^{-1}v$ 
(given as a $\GL(r, \CC)$-orbit in $\Hom( \pi_1(\rho^{-1}v), \GL (r, \CC))$) is locally constant  in  $v\in S$ in an evident  sense. Whence the terminology.

Let be given an isoholonomic $\rho$-connection $\nabla_\rho$ on $\Fcal$. We then have an open covering $\{V_\alpha\}_\alpha$ of 
$S$ and for every $\alpha$  a flat holomorphic connection $\nabla^\alpha: \Fcal|\rho^{-1}V_\alpha\to \Omega_M\otimes \Fcal|\rho^{-1}V_\alpha$ which lifts $\nabla_\rho|: \rho^{-1}V_\alpha: \Fcal|\rho^{-1}V_\alpha\to \Omega_\rho\otimes \Fcal|\rho^{-1}V_\alpha$. 
If $\{\phi_\alpha\}_\alpha$ is a $C^\infty$ partition of unity on $S$ with $\sup (\phi_\alpha)\subseteq V_\alpha$, then 
$\nabla:=\sum_\alpha \rho^*(\phi_\alpha) \nabla^\alpha$ is a $C^\infty$-connection on $\Fcal$ which globally lifts $\nabla_\rho$ in a particular way: in terms of a local trivialization in the atlas described  above, this connection is given by a matrix of $\rho$-basic forms of type $(1,0)$. Its curvature is therefore
given by a matrix of $\rho$-basic 2-forms of Hodge level $\ge 1$ (i.e., is a sum of  forms of type $(2,0)$ and $(1,1)$). Hence the Chern form $\Chern_k(\Fcal, \nabla)$ is  a $\rho$-basic $2k$-form of Hodge level $\ge k$. Note that this remains so if we alter the connection $\nabla_\rho$ by adding to it a nilpotent relative differential $\eta$ (i.e., a section of $\Omega_{\rho}\otimes \Ecal\!nd(\Fcal)$ that takes values in the nilpotent endomorphisms), for then the curvature form of $\nabla$ will be nilpotent along the fibers of $\rho$ and so $\Chern_k(\Fcal, \nabla+\eta)$ will map to zero in $\Omega^{2k}_\rho$. Since  $\Chern_k(\Fcal, \nabla+\eta)$ is also closed, it is then still $\rho$-basic. So when the general fiber of $\rho$ is connected, $\Chern_k(\Fcal, \nabla+\eta)$ is the pull-back of one on $S$. In particular, the complex $k$th Chern class  of $\Fcal$ lies in the image of $\rho^*: H^{2k}(S; \CC)\to H^{2k}(M;\CC)$. 

But as we will see in our main application, 
it is possible for $\FF$ to have nontrivial holonomy along the fibers of $\rho$ and so $\FF$ need not be a sheaf  pull-back of a 
holomorphic vector bundle on $S$. 
\\

We next extend this to a stratified setting. This naturally leads us to consider `germ versions' of the notions we just introduced.
Let $X$ be a complex-analytic variety endowed with a stratification $\Scal$, by which we mean a finite partition of $X$ into connected nonsingular locally closed subvarieties, called \emph{strata}, such that the closure of a stratum is a subvariety that is a  union of strata. We partially order the collection of strata by letting  $S'\le S$ mean  that $S'\subseteq\overline S$.   

\begin{definition}\label{def:basic2}
A \emph{system of retractions $\rho=(\rho_S)_S$} for $(X, \Scal)$ assigns to each $S\in\Scal$  an analytic retraction $\rho_S :X_S\to S$  with the property that  when $S'<S$, then $\rho_S'\rho_S=\rho_S'$ on $X_{S'}$. We then say that $(X, \Scal, \rho)$ is a  \emph{rigidified} stratified variety\footnote{There is of course also a $C^\infty$-variant of this notion, but we will here be only interested in the holomorphic version.}.
\end{definition}

Here $X_S$ denotes the germ of $X$ at $S$ and so this means that $\rho_S$  is represented by an analytic retraction whose domain is  an unspecified  neighborhood $U_S$ of $S$ in $X$. If the stratification $\Scal$ satisfies Whitney's $(a)$ condition,  then we  may  take $U_S$ so small such that  for every $S'\in \Scal$,  $\rho_S|U_S\cap S'$ is a submersion.  Note that for every  $S\in\Scal$,  the collection $\{\rho_{S'}|\overline S\}_{S'<S}$ is a system of retractions for $(\overline S, \Scal |\overline S)$. A complex submanifold of a complex manifold need not  be a holomorphic retract of some neighborhood of it and so the mere existence of such a system indicates that the stratification is  quite special. A standard example is the natural stratification of a torus embedding. We will see that the Satake and toric compactifications of a locally symmetric variety also come with this structure.

We make the rigidified stratified varieties objects of a category: a morphism $(\tilde X, \tilde\Scal, \tilde\rho)\to (X, \Scal, \rho)$ is given by a complex-analytic morphism $\pi :\tilde X\to X$ that takes any stratum $\tilde S$ of $\tilde\Scal$ submersively  to a stratum $S$ of $\Scal$ in such a manner that on a neighborhood of $\tilde S$ we have $\pi\tilde\rho_{\tilde S}=\rho_S\pi$ and we  demand that the preimage  of the union of the  open strata in $X$ is equal to the union of the  open strata in $\tilde X$.

\emph{From now on  $(X, \Scal, \rho)$ is a rigidified stratified variety with $X$ topologically normal} (in the sense that normalization is a homeomorphism). We denote by $\mathring{X}\subseteq X$ the union of the  open strata and by  $j: \mathring{X}\subseteq X$ and $i_S: S\subseteq X$ ($S\in\Scal$) the inclusions. We write $\mathring{\rho}_S$  for the restriction of $\rho_S$ to $\mathring{X}$.  The assumption of topological normality  guarantees that a general fiber  of $\mathring{\rho}_S$ is connected.

\begin{definition}\label{def:basic3}
We say that a $C^\infty$-differential form on $\mathring{X}$ is 
\emph{$\rho$-basic} if for every stratum $S\in\Scal$, its germ at $S$ is $\mathring{\rho}_S$-basic. 
\end{definition}

The $\rho$-basic $C^\infty$-differential forms in $j_*\Ascr^\pt_{\mathring{X}}$ make up a differential (bigraded) subalgebra  $\Ascr^\pt_{X, \rho}$ that is a fine resolution of the constant sheaf $\CC_X$ on $X$ (see Verona \cite{verona} and  Theorem 4.2 in \cite{gp}). It has  the property that for all $S\in\Scal$, 
$i_S^{-1}\Ascr^\pt_{X, \rho}=\Ascr^\pt_S$.  Its holomorphic part defines a subcomplex  $\Omega^\pt_{X, \rho}\subseteq j_*\Omega^\pt_{\mathring{X}}$ with a similar property: $i_S^{-1}\Omega^\pt_{X, \rho}=\Omega^\pt_S$. This is also a resolution of the  constant sheaf $\CC_X$ and we can regard  $\Ascr^\pt_{X, \rho}$ as a double complex which resolves it. So $(\Ascr^{p,\pt}_{X, \rho}, \bar\partial)$ resolves $\Omega^p_{X, \rho}$ and we have a Hodge-De Rham  spectral sequence
\[
E^{p,q}_2=H^q(X, \Omega^p_{X, \rho})\Rightarrow H^{p+q}(X, \CC).
\]
If we are in the complex projective setting, then one may wonder whether this spectral sequence degenerates and yields the Hodge filtration of the mixed Hodge  structure on $H^\pt (X)$. This is not so in general (there exist examples for which $H^0(X, \Omega^p_{X, \rho})=0\not=F^pH^p(X)$), but if it is  at least true that the limit filtration of this spectral sequence refines the Hodge filtration, then we would have a generalization of Theorem \ref{thm:main} below and would probably also end up with a simpler proof of it.

Note that a morphism $\pi: (\tilde X, \tilde\Scal, \tilde\rho)\to (X, \Scal, \rho)$ determines a map of sheaf complexes $\pi^{-1}\Ascr^\pt_{X, \rho}\to \Ascr^\pt_{\tilde X, \tilde\rho}$ which induces the usual map $\pi^*: H^\pt (X; \CC)\to  H^\pt (\tilde X; \CC)$ on cohomology.

We extend the notions introduced in \ref{def:basic1} to this stratified germ type of setting:

\begin{definition}\label{def:basic3}
Let $\Fcal$ be a holomorphic vector bundle on $\mathring{X}$.  A \emph{flat $\rho$-connection}  on $\Fcal$ assigns to every $S\in\Scal$ a flat $\rho_S$-connection $\nabla_{\rho_S}$ on  $i_S^{-1}j_*\Fcal$ (it is then denoted $\nabla_\rho$) and is subject to a compatibility condition up to nilpotents:
noting that for any pair of incident strata $S\ge S'$, the connection $\nabla_{\rho_{S'}}$ induces a flat connection along the fibers of  $\mathring{\rho}_S$ on the common domain of $\mathring{\rho}_S$ and  $\mathring{\rho}_{S'}$ (for the submersion $\mathring{\rho}_{S'}$ there factors through $\mathring{\rho}_S$) so that we may write this connection there as $\nabla_{\rho_S}+\eta^{S'}_S$ with $\eta^{S'}_S$ a section of $\Omega_{\rho_S}\otimes\Ecal\! nd(\Fcal)$, then we require that
$\eta^{S'}_S$ is a \emph{nilpotent} relative differential. More generally, whenever we have a chain of strata $S>S_1>\cdots> S_n$, we ask that the $\eta^{S_1}_{S}, \dots ,\eta^{S_n}_{S} $ span on their common domain of definition a complex vector space  of nilpotent relative differentials.

We say that the flat $\rho$-connection on $\Fcal$ is \emph{isoholonomic} if  for every $S\in \Scal$, $\nabla_{\rho_S}$ is so on $i_S^{-1}j_*\Fcal$.
\end{definition}

Given a flat $\rho$-connection on $\Fcal$, then its $\rho$-flat local sections  define a subsheaf of $j_*\Fcal$, but this subsheaf can be zero on certain strata and is probably of little interest unless the holonomies are trivial.  More relevant to us  is the subsheaf of $\Ocal_{X,\rho}$-algebras $\mathcal{E}\!nd(\Fcal, \nabla_\rho)\subseteq j_*\mathcal{E}\!nd(\Fcal)$ of 
 $\rho$-flat local endomorphisms of $j_*\Fcal$, at least when the  flat $\rho$-connection on $\Fcal$ is isoholonomic, for then 
 $i_S^{-1} \mathcal{E}\!nd(\Fcal, \nabla_\rho)$ is  locally like $\mathcal{E}\!nd_{\Ocal_S}(\Ocal_S^r)$ (it is a sheaf of Azumaya $\Ocal_S$-algebras).

Note that if  $\pi: (\tilde X, \tilde\Scal, \tilde\rho)\to (X, \Scal, \rho)$ is a morphism of rigidified stratified, topological normal varieties, then a flat $\rho$-connection on $\Fcal$ determines one on the pull-back of $\Fcal$ along $\mathring{\pi}: \mathring{\tilde X}\to \mathring{X}$ (that we simply denote by $\pi^*\nabla_\rho$): if $\pi$ maps $\tilde S\in \tilde\Scal$ to $S\in\Scal$, then $\nabla_{\rho_S}$ determines in an obvious manner a flat connection along the fibers of  $\mathring{\tilde\rho}_{\tilde S}$ and the resulting system has the required properties. It is isoholonomic when $\nabla_\rho$ is.

\begin{proposition}\label{prop:main1}
With a holomorphic vector bundle $\Fcal$ on $\mathring{X}$ that is endowed with an isoholonomic flat $\rho$-connection $\nabla_\rho=(\nabla_{\rho_S})_{S\in\Scal}$
is associated a complex Chern class lift  $\chern_k(\Fcal, \nabla_\rho)\in H^{2k}(X; \CC)$ of $\chern_k(\Fcal)_\CC\in H^{2k}(\mathring{X}; \CC)$ ($k\ge 0$), which is functorial in the sense that  if $\pi: (\tilde X, \tilde S, \tilde\rho)\to (X, \Scal, \rho)$ is a morphism of rigidified stratified spaces, then 
$\pi^*\chern_k(\Fcal, \nabla_\rho)=\chern_k(\mathring{\pi}^*\Fcal, \tilde\nabla_{\tilde\rho})$. It also has the property that if the relative holonomy of $\nabla_{\rho_S}$ is trivial at every point of $S$ and for any pair of incident strata $S\ge S'$, $\eta_S^{S'}=0$,
then $(\Fcal, \rho)$ extends naturally  to $(\hat\Fcal, \rho)$ on $X$ (as a holomorphic vector  bundle with flat connections along the retractions) and  $\chern_k(\Fcal, \nabla_\rho)=\chern_k(\hat\Fcal)_\CC$.
\end{proposition}
\begin{proof}
By assumption $X$ admits a covering by open subsets $U_\alpha$ with the property that there is unique $S_\alpha\in\Scal$ such that $S_\alpha\cap U_\alpha$ is closed, $\rho_{S_\alpha}$ and $\nabla_{\rho_{S_\alpha}}$ are defined on $U_\alpha$ and $\nabla_{\rho_{S_\alpha}}$ lifts to a flat holomorphic connection $\nabla^\alpha$ on $\Fcal |\mathring{U}_\alpha$. We can choose a partition of unity $\{\phi_\alpha :X\to [0,1]\}_\alpha$ with $\sup(\phi_\alpha)\subseteq U_\alpha$ and with $\phi_\alpha |U_\alpha$  factoring through $\rho_{S_\alpha}$. Then $\nabla:=\sum_\alpha\phi_\alpha\nabla^\alpha$ is a $C^\infty$-connection on $\Fcal$  with the property that any $S\in\Scal$  admits a neighborhood $U_S$ in $X$ such that  for any chain $S>S_1>\cdots >S_n$ in $\Scal$,
the relative connection that $\nabla$ induces along $\rho_{S}|\mathring{U}_S\cap U_{S_1}\cap\cdots U_{S_n}$ is a convex linear combination of 
$\nabla_{\rho_{S}}$ and $\eta^{S_1}_{S}, \dots, \eta^{S_n}_{S}$. This implies that the Chern form $\Chern_k(\Fcal,\nabla)$ is a $\rho$-basic closed form of Hodge level $\ge k$. By the fine resolution property cited above, it therefore defines a cohomology class $\chern_k(\Fcal,\nabla)$. 

The proof that this class is independent of our choices is a straightforward generalization of the standard proof and is based on the observation that the $C^\infty$-connections on $\Fcal$ satisfying the above property is an affine space. Indeed,  if ${}'\nabla$ is another such connection, then we define on the pull-back $\hat\Fcal:=pr_{\mathring{X}}^*\Fcal$ of $\Fcal$ along
$pr_{\mathring{X}}: \CC\times \mathring{X}\to \mathring{X}$ a connection $\hat\nabla$ given  on a pulled back section $pr_{\mathring{X}}^*s$ as  $(1-t)\nabla(s)+t\,{}'\nabla(s)$.  Then $\Chern_k(\hat\Fcal,\hat\nabla)$ defines (by the result above) a class  $\chern_k(\hat\Fcal, \hat\nabla)\in H^{2k}(\CC\times X)$. This class evidently restricts to  $\chern_k(\Fcal, \nabla)$ resp.\  $\chern_k(\Fcal, {}'\nabla)$ if we take the first coordinate $0$ resp.\ $1$ and so $\chern_k(\Fcal, \nabla)=\chern_k(\Fcal, {}'\nabla)$. It is straightforward to verify that these Chern classes have the asserted naturality behavior. 

Assume now that  the relative holonomy of $\nabla_{\rho_S}$ is trivial at every point of $S$ and for any pair of incident strata $S\ge S'$, $\eta_S^{S'}=0$.
Choose for every stratum $S$ a neighborhood $U_S$ of $S$  in $X$ contained in the domain of $\rho_S$ and $\nabla_{\rho_{S}}$ such that  
$\nabla_{\rho_{S}}$ has no holonomy on $\mathring{U}$. Then the subsheaf of $\Fcal|\mathring{U}_S$ of $\nabla_{\rho_S}$-flat sections has a direct image on $U_S$ whose restriction to $S$ is a holomorphic vector bundle. Its pull-back as a vector bundle along $\rho_S$ can on a neighborhood of $S$ in $\mathring{X}$ be identified with $\Fcal$ and so this defines an extension of $\Fcal$ across $U_S$. Since the $\eta_S^{S'}$ vanish, such extensions agree on overlaps. It follows that $\Fcal$ extends to a holomorphic bundle $\hat \Fcal$ on $X$. Although $X$ may be singular, a connection as constructed above extends to a connection $\hat\nabla$ on $\hat\Fcal$ in the sense that  it is locally given by a matrix with entries in $\Omega_{X}$ (so restrictions of holomorphic differentials on an ambient complex manifold). Then $\Chern_k(\hat\Fcal, \hat\nabla)$ is a $C^\infty$ $2k$-form (i.e.,  locally the restriction to $X$ of a form defined on an ambient $C^\infty$-manifold) and therefore defines a class in $H^{2k}(X;\CC)$. Since its restriction to $\mathring{X}$ is $\Chern_k(\Fcal, \nabla)$, this class is in fact $\chern_k(\Fcal, \nabla_\rho)$. 
\end{proof}

In the situation of the last clause of Proposition \ref{prop:main1}  we find that $\chern_k(\Fcal, \nabla_\rho)$ lifts to an integral class. But we will see that this is not so in general.

\begin{remark}\label{rem:chat}
We shall later want to work with Chern characters $ch_k$ rather then with Chern classes $c_k$. Since we always use $\QQ$-vector spaces as coefficients,  there is no loss of information here: $ch_k$ is a universal polynomial of weighted degree $k$ with rational coefficients in $c_1, \dots , c_k$ and vice versa.  
These Chern characters can also be obtained via an Atiyah class, which is perhaps closer in the spirit of algebraic geometry, albeit that they then come to be realized as De Rham classes.
 An isoholonomic flat $\rho$-connection  defines a natural lift of the  Atiyah class of $\Fcal$,  
$\aty (\Fcal)\in H^1(\mathring{X}, \Omega^1_{\mathring{X}}\otimes \mathcal{E}\!nd(\Fcal))$ to an element $\aty (\Fcal, \nabla_\rho)\in H^1(X, \Omega^1_{X, \rho}\otimes \mathcal{E}\!nd(\Fcal, \nabla_\rho))$. A representative  as a 1-\v Cech cocycle is obtained from the collection $(U_\alpha,\nabla^\alpha)_\alpha$  in the proof above: it is given by $U_{\alpha\beta}=U_\alpha\cap U_\beta  \mapsto \nabla^\beta-\nabla^\alpha$. We then define the  \emph{twisted Goresky-Pardon Chern character}  as the image of $\aty (\Fcal, \nabla_\rho)$ under the map
\begin{align*}
H^1(X, \Omega^1_{X, \rho}\otimes \mathcal{E}\!nd(\Fcal, \nabla_\rho))&\to \oplus_{k=0}^\infty H^k(X, \Omega^k_{X, \rho}),\\  
A&\mapsto \Tr(\exp (-A))=\sum_{k=0}^\infty\frac{\Tr \big((-A)^{\cup k}\big)}{k!}.
\end{align*}
This class is closed for all the differentials in the  Hodge-De Rham spectral sequence and then yields $(2\pi \sqrt{-1})^k \cha_k(\Fcal, \nabla_\rho)$.
This observation leads us to:
\end{remark}

\begin{corollary}\label{cor:real}
Suppose that  in the situation of Proposition \ref{prop:main1}, the setting is algebraic over $\RR$, that is, $X$, its stratification and the retractions
appearing there and the vector bundle $\Fcal$ are defined over $\RR$. Then the twisted Goresky-Pardon Chern character $(2\pi \sqrt{-1})^k \cha_k(\Fcal, \nabla_\rho)$ is fixed under full complex 
conjugation (acting on both $X$ and the coefficient field $\CC$).
\end{corollary}
\begin{proof}
We must verify  that 
$(2\pi \sqrt{-1})^k \cha_k(\Fcal, \nabla_\rho)$ is fixed under  the anti-linear map $z\in H^\pt (X; \CC)\mapsto \overline{\iota^*z}\in H^\pt (X; \CC)$, where
$\iota :X\to X$ is complex conjugation. In the \v Cech description of the Atiyah class above we can choose the collection $(U_\alpha,\nabla^\alpha)_\alpha$ in such a manner  that $\iota$ acts compatibly on our index set: $\iota(U_\alpha)=U_{\iota\alpha}$ and $\iota^*\nabla^{\iota\alpha}=\overline{\nabla}^{\alpha}$.
Then it is clear from the definition that $(2\pi \sqrt{-1})^k \cha_k(\Fcal, \nabla_\rho)$ has the asserted property.
\end{proof}

\begin{theorem}\label{thm:main}
Suppose that  in the situation of Proposition \ref{prop:main1}, the setting is algebraic  and  that $X$ is compact. Suppose moreover that the resolution $\pi:\tilde X\to X$ that satisfies the holonomy property with respect to $(\Fcal, \nabla_\rho)$ extends to a \emph{stratified resolution} in the sense below. Then $\chern_k(\Fcal, \nabla_\rho)$ is of Hodge level $\ge k$, i.e., lies in  $F^kH^{2k}(X;\CC)$. 
\end{theorem}

\begin{remark}\label{rem:chhodge}
Since the cup product is compatible with the Hodge filtration, it then follows that the corresponding class $\cha_k(\Fcal, \nabla_\rho)$ is also of Hodge level $\ge k$.
\end{remark}

The notion of a \emph{resolution  of a stratified variety} that appears in the formulation of the theorem above expresses the fact that such a variety is equisingular along strata in a rather strong sense. Among other things, it can be shown to imply Whitney's $(a)$ condition.
 We define this notion and prove the theorem in the next subsection.

\subsection*{Stratified resolutions}
We begin with noting that if on a complex manifold $Y$ is given  a normal crossing divisor $D$, then 
$Y$ acquires a natural stratification, where a stratum is a connected component of the locus where for some integer $l\ge 0$ exactly $l$ local branches of $D$ meet. With $D$ given, we will often write  $Y^{(l)}$ for the normalization of the locus where at least $l$ branches of $D$ meet (so 
 that $Y^{(0)}=Y$). This is clearly a complex manifold. If $E$ is a connected component of $Y^{(l)}$, then the locus where $>l$ branches of $D$ meet traces out on $E$ a normal crossing divisor, which is simple when $D$ is and whose normalization is contained in $Y^{(l+1)}$. When $l>0$, then  for the same reason, $E$ naturally maps to a number of connected components of $Y^{(l-1)}$. When $D$ is simple, this number is $l$ and the maps are embeddings. 
 
 Let $(X, \Scal)$ be a stratified analytic variety and assume that  the normalization of $X$ is a homeomorphism. 

\begin{definition}\label{def:stratares}
An \emph{$\Scal$-resolution} of $(X, \Scal)$ consists of giving  
for every stratum $S\in \Scal$ a resolution of its closure, $\pi_S: \tilde S\to \overline S$,  such that 
\begin{enumerate}
\item [(i)] $\pi_S: \tilde S\to \overline S$ is an isomorphism over $S$ and the preimage of $\partial S$ is a simple normal crossing divisor
$D_{\tilde S}$ (so that $\tilde S$ comes with a natural stratification),
\item [(ii)] when $S'< S$,  then $\tilde S[S']:= \overline{\pi_S^{-1}S'}$ is a union of irreducible components of  $D_{\tilde S}$ and  we have a factorization
\[
\pi_{S}: \tilde S[S']\xrightarrow{\pi^{S'}_S}\tilde S'\xrightarrow{\pi_{S'}}\overline S'
\]
that maps every stratum of $\tilde S[S']$ onto a stratum of $\tilde S'$ and 
\item [(iii)] when $S''< S'$, then $\pi^{S''}_{S}\big |\tilde S[S'']\cap  \tilde S[S']$ factors as
\[
\pi^{S''}_{S}: \tilde S[S'']\cap  \tilde S[S']\xrightarrow{\pi^{S'}_S} \tilde S'[S'']\xrightarrow{\pi^{S''}_{S'}}\tilde S''.
\]
\end{enumerate}
\end{definition}
Note that then any stratum $S\in\Scal$ inherits such a structure in the sense that  the collection $\{\pi_{S'}\}_{S'\le S}$ defines a $\Scal |\overline S$-resolution of $\overline S$.

In order to prove Theorem \ref{thm:main} we first show how the above notion gives rise to a simplicial resolution that can be used compute the
cohomology of $X$ and its mixed Hodge structure, when that makes sense. 

Obviously, the collection of $\pi_S: \tilde S\to \overline S$, where  $S\in\Scal$ runs over the open strata, defines a resolution $\pi :\tilde X\to X$ of $X$ whose exceptional set is a normal crossing divisor.  So $X$ 
can be regarded as a quotient space of $\tilde X$ with the identifications taking place over the strata of depth $\ge 1$.  Let $S>S'$  be a pair 
of incident strata whose depths differ by 1. When we regard $\overline S$ as a quotient of $\tilde S$, then the identification over $S'$ is exhibited
by $\pi^{S}_{S'} : \tilde S[S']\to \tilde S'$. In order to let all such identifications 
take place by means of morphisms between smooth varieties, it is best to replace $\tilde S[S']$ by its normalization. 
This means that we should do this for every connected component of this normalization.  It is then wise to remember that these 
connected components are glued to each other in  $\tilde S[S']$.  We may continue this process with any stratum of depth 2 and finally end up with a small category $\Sscr$ of compact complex manifolds over $X$ that  has $X$ as a direct limit in the category of topological spaces. Here is more precise description of $\Sscr$.

An object of $\Sscr$ is a connected component $E$ of $\tilde S^{(l)}$ for some $S\in\Scal$ and some $l\ge 0$. 
So when we regard  $E$ as a subvariety of $\tilde S$,  it is the  closure of a stratum. 
We describe two types of  basic morphisms $E\to E'$ between two such objects and stipulate that these generate the $\Sscr$-morphisms. 
The first one is when $E'$ is obtained from $E$ by forgetting one of the $l$ irreducible components of $D_{\tilde S}$ which contains $E$ 
and $E\to E'$ is the obvious embedding. The other is defined only when $E$ is contained in
$\tilde S[S']$ for some $S'<S$. Then $E':=\pi_S^{S'}(E)$ is  a connected component of $\tilde S'{}^{(l')}$ for some $l'$ and  the resulting map $E\to E'$ is the other type of  basic morphism. 

Now recall that the \emph{nerve} of the small  category $\Sscr$ is a simplicial set whose $n$-simplices are chains of length $n$: $E_\pt=(E_0\to E_1\to\cdots \to E_n)$ in $\Sscr$. We then obtain a simplicial  space $X_\pt$ by taking for $X_n$ the disjoint union of the objects $\first (E_\pt)$ where $E_\pt$ runs over the chains of length $n$ in $\Sscr$ and $\first (E_\pt)$ stands for the first term of such a chain. (So the connected components of $X_n$ are objects of  $\Sscr$, but the indexing is by the $n$-simplices, so that several copies of the same   $\Sscr$-object  may appear.)
Its geometric realization $|X_\pt|$ (which is obtained in a standard fashion as a quotient  of the disjoint union of the products $|\Delta^l|\times X_l$) is a space over $X$ and this structural map is a homotopy equivalence. It can be used for cohomological descent: the face maps $\partial_i: X_n\to X_{n-1}$,  $0\le i\le n$, are used  to define a double cochain complex
\[
C^\pt (X_\pt):  0\to C^\pt(X_0) \to C^\pt(X_1)\to\cdots
\]
and the obvious chain homomorphism from $C^\pt (X)$ to the associated simple complex $sC^\pt (X_\pt)$ induces an isomorphism on (integral) cohomology.
In particular, we have  spectral sequence
\begin{equation}
\label{eq:1}
E_1^{r,s}= H^s(X_r)\Rightarrow H^{r+s}(X). 
\end{equation}
We note that the edge homomorphism $H^\pt(X)\cong H^\pt(sC^\pt (X_\pt))\to H^\pt(X_0)=\oplus_{E\in \Sscr} H^\pt (E)$ is  induced by the obvious map $\sqcup_{E\in\Sscr}E =X_0\to X$. 

Suppose that we are in the algebraic setting so that varieties and  morphisms are complex-algebraic. Then   $H^\pt (X)$ carries a mixed Hodge structure and we can use this construction to identify that structure: the above spectral sequence is one of mixed Hodge structures. This implies that when $X$ is compact, it degenerates at $E_2$ (all higher differentials are zero since their source and target when nonzero have different weight) and yields the weight filtration:
\begin{equation}\label{eq:2}
\gr^W_sH^{r+s}(X;\QQ)=E_2^{r,s}=H\big(H^s\big(X_{r-1}; \QQ)\to H^s(X_r; \QQ)\to H^s(X_{r+1}; \QQ)\big).
\end{equation}
Moreover, if we use $A^\pt(M)$ to denote the $\CC$-valued De Rham complex of a complex algebraic  manifold $M$ and $F^\pt A^\pt(M)$ its Hodge filtration, then the Hodge filtration of $sA^\pt (X_\pt)$ defines the Hodge filtration of $H^\pt(X;\CC)$.

Perhaps the simplest nontrivial example is when $X$ has only two strata: $S$ and $X-S$ and $\pi:\tilde X\to X$ is a resolution  with $\pi^{-1}S$ nonsingular.
Then the complex $C^\pt(X_\pt)$ is just $0\to C^\pt(\tilde X)\oplus C^\pt(S)\to  C^\pt(\pi^{-1}S)\to 0$ and  the associated exact sequence
\[
\cdots \to H^{s-1}(\pi^{-1}S)\to H^{s}(X)\to H^s(\tilde X)\oplus H^s(\pi^{-1}S)\to H^s(\pi^{-1}S)\to\cdots
\]
yields the weight filtration:  $W_{s-2}H^{s}(X)=0$,  $W_{s-1}H^{s}(X)$ is the image of the map $H^{s-1}(\pi^{-1}S; \QQ)\to H^{s}(X;\QQ)$ and  $W_sH^{s}(X)=H^{s}(X;\QQ)$.

\begin{proof}[Proof of Theorem \ref{thm:main}]
For $\nabla$ as constructed in the proof of Proposition \ref{prop:main1}, the Chern form $\Chern_k(\Fcal, \nabla)$ defines a closed $2k$-form on $\tilde S$ for every $S\in\Scal$. This form is of Hodge level $\ge k$. It restricts to a  $2k$-form $\Chern_k(\Fcal_{|E},\nabla_E)$ on every $\Sscr$-object $E$ with the same property 
and for every  $\Sscr_\rho$-morphism  $\phi: E\to E'$ we have $\phi^* \Chern_k(\Fcal_{|E'},\nabla_{E'})= \Chern_k(\Fcal_{|E},\nabla_{E})$. This means that $(\Chern_k(\Fcal_{|E},\nabla_{E}))_E$ defines a cocycle of degree $2k$
in $A^\pt(X_\pt)$ and thus defines a class $\chern_k(\Fcal, \nabla)\in F^kH^{2k}(X)$. 
\end{proof}

\section{A first application to Baily-Borel compactifications}

\subsection*{Review of the Baily-Borel compactification}
Let $\Gcal$ be a connected reductive complex algebraic group that is defined over $\RR$. 
Write $G$ resp.\  $G_\CC$ for $\Gcal (\RR)$ resp.\ $\Gcal (\CC)$ endowed with the Hausdorff topology.  We assume that $G$ has compact center and that the symmetric space $\XX$ of $G$ is endowed with a $G$-invariant complex structure. To say that $\XX$ is the 
symmetric space $\XX$ of $G$ means  that for every $x\in \XX$ the stabilizer $G_x$ is a maximal compact subgroup of $G$ and to say 
that $\XX$ comes with a $G$-invariant complex structure amounts to the property that $G_x$  contains an embedded a copy of the circle group $\Un (1)$ in its center whose action on $T_x\XX$ defines its complex structure (it is a nontrivial action by scalars of unit norm). This makes $\XX$ a bounded symmetric domain.  It appears naturally as an open $G$-orbit in a complex projective manifold $\check{\XX}$, called  the \emph{compact dual of $\XX$}, on which $G_\CC$ acts transitively. It has the property that for $x\in \XX$, the  $G_{\CC}$-stabilizer $G_{\CC, x}$ is simply the  complexification of $G_x$.

We now assume that $\Gcal$ is defined over $\QQ$  and let $\Gamma\subset G(\QQ)$ be an arithmetic subgroup. For what follows the passage to  a subgroup of $\G$ of finite index will be harmless, and so we will assume from the outset that $\G$ is \emph{neat}. This means that for every finite dimensional representation  $\rho: G_\CC\to \GL (n,\CC)$  the subgroup of $\CC^\times$ generated by the eigenvalues  of elements of $\rho (\G)$ has no torsion (actually it suffices to verify this  for just one faithful representation). This  implies that the arithmetically defined  subquotients of $\G$  are torsion free.  The action of $\G$ on $\XX$ is then proper and free so that the orbit space $\XX_\G$ is  complex manifold. The Baily-Borel compactification, which we will presently recall, shows that  $\XX_\G$ has even the structure of nonsingular quasi-projective variety. 

A central role in the Baily-Borel theory is played by the collection $\Pscr_{\max}=\Pscr_{\max}(G)$ of maximal proper parabolic subgroups of $\Gcal$ defined over $\QQ$ and so let us fix some $P\in \Pscr_{\max}$. We review the structure of  $P$ and the way it acts on $\XX$. 
Its unipotent radical  $R_u(P)\subseteq P$ is at most 2-step unipotent: if
 $U_P\subseteq R_u(P)$ denotes its center (a nontrivial vector group), then $V_P:=R_u(P)/U_P$ is also a (possibly trivial) vector group. Adopting the convention to denote the associated Lie algebras by the corresponding Fraktur font, then the  Lie bracket defines an antisymmetric bilinear map $\vfrak_P\times \vfrak_P\to \ufrak_P$. This map is  equivariant with respect to the adjoint  action of $P$ on these vector spaces. Note 
 that $P$ acts on $\ufrak_P$ and $\vfrak_P$  through its Levi quotient $L_P:=P/R_u(P)$.
The reductive group $L_P$ has in $\ufrak_P$ a distinguished  open orbit that is a strictly convex cone $C_P$ having the property that if we exponentiate  $\sqrt{-1}C_P$ to a semigroup in $G_\CC$, then this semigroup leaves $\XX$ invariant (think of the upper half plane in $\CC$ that is invariant under the  semigroup of translations in $\sqrt{-1}\RR_{>0}$). The $G$-stabilizer of  $\ufrak_P$  is $P$ and so  $\ufrak_P$ determines $P$. This gives rise to a partial order $\le$ on $\Pscr_{\max}$ by stipulating that $Q\le P$ in case $\ufrak_Q\subseteq\ufrak_P$ (in $\gfrak$). This last property is equivalent to $C_Q\subseteq \overline C_P$. The other (non-maximal) $\QQ$-parabolic subgroups of $G$ are obtained from chains in $\Pscr_{\max}$:
for a chain $P_0<P_1<\cdots <P_n$ in $\Pscr_{\max}$, $P:=P_0\cap \cdots \cap P_n$ is a $\QQ$-parabolic subgroup. It has the property that its unipotent radical contains the unipotent radical of each $P_i$: $\cup_{i=1}^n R_u(P_i)\subseteq R_u(P)$.

For $P\in \Pscr_{\max}$, the $\QQ$-split center $A_P$ of $L_P$  is isomorphic as such to the multiplicative group (and so $A_P\cong \RR^\times$). It acts on $\vfrak_P$ by a faithful character (multiplication by scalars) and on $\ufrak_P$ by the square of that character (so that it indeed preserves $C_P$).  The \emph{horizontal subgroup} $M^\hor_P\subseteq L_P$ (for some authors  the superscript stands for \emph{hermitian})  is the kernel of the action of $L_P$ on $\ufrak_P$. This is a reductive subgroup defined over $\QQ$ with compact center. 
The  centralizer of $M^\hor_P$ in  $L_P$ is a reductive $\QQ$-subgroup whose commutator subgroup we denote by $M^\ver_P$ (we like to think that the symbol $\ell$ should refer to \emph{link} rather than \emph{linear}---the explanation for this  terminology will become clear below).
This group acts in such a manner on the real projectivization of $C_P$ (in the projective space of $\ufrak_P$) that the latter is the symmetric space of  $M^\ver_P$. So we may regard $C_P$ as the symmetric space of $L^\ver_P:=M^\ver_P.A_P$. The group $L^\ver_P$ supplements  $M^\hor_P$ in $L_P$ up to a finite  central subgroup. We denote the preimage of $L^\ver_P$ in $P$ by $P^\ver$.  

The action of $P$ on $\XX$ is still transitive. Important for what follows is that the formation of the $P^\ver$-orbit space of $\XX$ remains in the holomorphic category: it defines a holomorphic submersion of complex manifolds $\XX\to \XX(P)$, with $\XX(P)$ appearing as 
the symmetric domain of $M^\hor_P$. This is called a \emph{rational boundary component} of $\XX$.  The $P^\ver$-orbits in  $\XX$
(so the fibers of   $\XX\to \XX(P)$) are also orbits of the 
semi-subgroup $R_u(P)+\exp (\sqrt {-1}C_P)\subseteq G_\CC$ in $\check{\XX}$ and  this description is essentially an abstract way of a realizing $\XX$ as a Siegel domain of the third kind. To be precise, we have a natural factorization
of $\rho_P: \XX\to\XX(P)$:
\begin{equation}\label{eq:3}
\begin{CD}
\rho_P: \XX@>{\rho'_P}>> \XX(P)'@>{\rho''_P}>>  \XX(P), 
\end{CD}
\end{equation}
where is the first map $\rho'_P$ is a bundle of tube domains (a `torsor' over $ \XX(P)'$ for the semigroup $\exp(\ufrak+\sqrt {-1}C_P))$) and the second map
$\rho''_P$ is a principal bundle of the vector group $V_P=\exp (\vfrak_P)$ (so a bundle of affine spaces). The latter has also the structure of a complex affine space bundle, but beware that  this complex structure on a fibre (which can be given  as a complex structure on its translation space $\vfrak$) will in general vary with the base point. The map $M^\hor_P\to L_P/L_P^\ver=P/P^\ver$ is an isogeny: it is onto and has finite  kernel. We write $G_P$ for $P/P^\ver$.  The action of $M^\hor_P$ on $\XX(P)$ is through this quotient and we prefer to regard  $\XX(P)$ as the symmetric space of  the quotient $G_P$ of $P$ rather than of  the subquotient $M^\hor_P$  of $P$ (see the example of the symplectic group below). 

Every  $Q\in \Pscr_{\max}(G)$ with $Q>P$ has by definition  the property that $\ufrak_Q\supset \ufrak_P$. But it is then even true that $Q^\ver\supset P^\ver$ and so the projection $\rho_Q: \XX\to \XX(Q)$ factors through $\rho_P: \XX\to \XX(P)$ via a morphism that we shall denote by $\rho_{Q_{/P}}:  \XX(P)\to \XX(Q)$. The latter can be understood as the formation of a rational boundary component of $\XX(P)$. Indeed, $Q$
defines a maximal proper parabolic subgroup of $G_P$, namely the image of $Q\cap L_P$ in $L_P/L_P^\ver=G_P$
(that we shall denote by $Q_{/P}$). This identifies $\Pscr_{\max}(G_P)$ as a  partially ordered  set with $\Pscr_{\max}(G)_{>P}$. The unipotent radical 
of $Q_{/P}$ is the image of $R_u(P)\cap R_u(Q)$ in $Q_{/P}$. Its center $U_{Q_{/P}}$ is the image of $U_Q$, $U_Q/U_Q\cap R_u(P)$. Similarly, the cone $C_{Q_{/P}}\subseteq \ufrak_{Q_{/P}}$ is the image of $C_Q\subseteq \ufrak_Q$ under the projection $\ufrak_Q\to \ufrak_Q/ \ufrak_Q\cap R_u(\pfrak)\cong \ufrak_{Q_{/P}}$:
\[
\begin{CD}
\ufrak_P \subseteq \ufrak_Q@>>> \ufrak_Q/ \ufrak_Q\cap R_u(\pfrak)\cong \ufrak_{Q_{/P}}\\
\cup &&\cup\\
C_P\le C_Q@>>> C_{Q_{/P}}.
\end{CD}
\] 
We define the \emph{Satake extension} of $\XX$ as a ringed space. As a set it is the disjoint union
\[
\XX^\bb:=\XX\sqcup \bigsqcup_{P\in\Pscr_{\max}}\XX(P).
\]
It is  endowed with the \emph{horocyclic topology}: the topology
generated by the open subsets of $\XX$ and the subsets $\Omega^{\bb_P}$, where  $P\in\Pscr_{\max}$ and  $\Omega\subseteq \XX$ is open and  invariant under both the semigroup $\sqrt{-1}C_P$ and the group  $\G\cap P^\ver$, and 
\[
\Omega^{\bb_P}:=\Omega \sqcup \bigsqcup_{Q\in\Pscr_{\max}; Q\le P} \rho_Q(\Omega).
\]
Since $\G\cap R_u(P)$ is cocompact in $R_u(P)$, we may replace here invariance under $\G\cap P^\ver$ by  invariance under $R_u(P).(\G\cap P^\ver)$ (but not in general by  invariance under $P^\ver$). Yet this topology is independent of $\G$: it only depends on  the $\QQ$-structure on $\Gcal$. This construction is natural in the sense that  the closure of any rational boundary component in  $\XX^\bb$ can be identified with its Satake extension. The  structure sheaf $\Ocal_{\XX^\bb}$ is the sheaf of complex-valued continuous functions that are holomorphic on every stratum $\XX(P)$. It is clear that $\G$ acts on this ringed space. The main theorem of Baily-Borel asserts among other things that the orbit space $(\XX^\bb_\G,\Ocal_{\XX^\bb_\G})$ is as a ringed space a normal compact analytic space that underlies the structure of normal projective variety and by a theorem of Chow, this projective structure is then unique. Moreover, the 
decomposition of $\XX^\bb$ into $\XX$ and its rational boundary components defines a decomposition  of $\XX^\bb_\G$ into nonsingular subvarieties (strata) such that the closure of any of these is a union thereof.  Any stratum is of the same type as $\XX_\G$: it is of the form $\XX(P)_{\G(G_P)}$ and hence has its own  Baily-Borel compactification. The preceding shows that the Baily-Borel compactification of a stratum maps  homeomorphically onto 
its closure in $\XX_\G^\bb$. This map is also morphism of varieties (that could be an isomorphism, but it is conceivable that this closure is not normal). This shows among other things:

\begin{corollary}\label{cor:}
The retractions $\{\rho_P: \XX\to \XX(P)\}_{P\in\Pscr_{\max}}$ endow the Baily-Borel stratification of  $\XX_\G^\bb$
with a natural system $\rho_\G^\bb$ of retractions, thus making it a rigidified stratified space.
\end{corollary}

\subsection*{Satake extension  of automorphic bundles}
Let $\Fcal$ be  an \emph{automorphic vector bundle} on $\XX$, that is, a complex vector bundle on $\XX$ endowed with a $G$-action lifting the one on $\XX$ in such a manner that for some (and hence for any) $x\in \XX$ the copy of  $\Un (1)$ in the stabilizer $G_x$ acts also complex  linearly on the fiber $\Fcal (x)$. Such a vector bundle is completely given by the action of $G_x$ on the complex vector space $\Fcal (x)$ and conversely, any finite dimensional complex representation of $G_x$ defines such a vector bundle. The bundle  $\Fcal$ 
 with  its $G$-action extends to the compact dual $\check{\XX}$ as a vector bundle with $G_\CC$-action and this extension (which we denote by 
 $\check\Fcal$) is unique. This is because the $G_x$-action on the complex vector space $\Fcal (x)$ extends 
to one of the complexification $G_{x, \CC}$ of $G_x$, and $G_{x, \CC}$ is just the $G_\CC$-stabilizer of $x$.  Since  the  
$G_\CC$-bundle $\check\Fcal$ is defined in the holomorphic category, it follows that $\Fcal$ comes with a  $G$-invariant holomorphic structure. 

Given $x\in \XX$, the compactness of $G_x$ implies  that $G_x$ leaves invariant an inner product in the fiber $\Fcal (x)$. This inner product then extends in a unique manner to a $G$-invariant inner product $h$ on $\Fcal$.  As is well-known, we then have a unique hermitian connection $\nabla$ on $(\Fcal, h)$ whose $(0,1)$-part is zero on local holomorphic sections. This connection  is of course also $G$-invariant. It is in fact independent of $h$. This is clear when $\Fcal (x)$ is irreducible as a representation of $G_x$, for then the inner product is unique up to  scalar and the general case then follows from this by decomposing  $\Fcal (x)$ into irreducible  subrepresentations. 
So we have canonically associated Chern forms $\Chern_n(\Fcal)=\Chern_n(\Fcal,\nabla)$ on $\XX$. Such a form is harmonic relative to a $G$-invariant metric on $\XX$, is $G$-invariant and of Hodge bidegree $(n,n)$. 
The $G$-equivariance allows us to descend all of this to $\XX_\G$, so that we get a holomorphic bundle $\Fcal_\G$ with connection on $\XX_\G$ whose Chern forms pull back to the ones of $(\Fcal, \nabla)$. The $G$-invariant connection $\nabla$ will in general not extend to $\check\Fcal$ and neither will the associated Chern forms.

\begin{lemma}\label{lemma:bbflatness}
The action of the semigroup $R_u(P)\exp({\sqrt{-1}C_P})$ on $\XX$  defines a natural flat $\rho_P$-connection on $\Fcal$. This identifies 
$\Fcal$ with the $\rho_P$-pull-back of a vector bundle $\Fcal(P)$ on $\XX(P)$ with $P/R_u(P)$-action (lifting the obvious $P/R_u(P)$-action on
$\XX(P)$). In particular,  $\Fcal(P)$ is automorphic relative the  $M_P^\hor$-action on $\XX(P)$ with $L_P^\ver$ acting   (possibly nontrivially) as a group of bundle automorphisms over $\XX(P)$. 

Finally, for any chain $P<P_1<\cdots <P_n$ in $\Pscr_{\max}$, the
flat $\rho_{P_k}$-connection on $\Fcal$, when regarded as a  $\rho_{P}$-connection, differs from the flat $\rho_{P}$-connection by a differential that takes its values in 
the nilpotent Lie algebra  $R_u(\pfrak\cap\pfrak_1\cap\dots \cap \pfrak_n)$.
\end{lemma}
\begin{proof}
Recall that the morphism $\rho_P: \XX\to \XX(P)$ is a principal bundle for the semisubgroup $R_u(P)\exp({\sqrt{-1}C_P})$ of $G_\CC$ and so  the restriction of $\Fcal$ is canonically trivialized as a complex vector bundle along the fibers of $\rho_P$. This trivialization can be made holomorphic. We show this by means of the factorization (\ref{eq:3}): the fibers of the first factor $\rho'_P: \XX\to \XX(P)'$ are orbits of the semisubgroup $\exp(\ufrak_P +\sqrt{-1}C_P)\subseteq U_P(\CC)$ and so over such orbits we get a holomorphic trivialization: we end up with a vector bundle 
$\Fcal'$ on $\XX(P)'$ such that $\Fcal$ is identified with the pull-back of $\Fcal'$ along $\rho_P'$.

The second factor $\rho''_P: \XX(P)'\to \XX(P)$ is a torsor for the vector group $V_P=\exp(\vfrak_P)$. Thus the trivial vector bundle over $\XX(P)$ with fiber $V_P$ has a holomorphic structure  yielding a holomorphic vector bundle over $\XX(P)$ (whose total space we denote by $\VV_P$) such that for every holomorphic local section $\sigma$ of $\XX(P)'\to \XX(P)$  with domain $N$, the map  $\VV_{P}|_N\to \XX(P)'|_{N}$, $v_z\mapsto v_z+\sigma (z)$ is biholomorphic. Any  trivialization of $\VV_P|_N$ then yields a holomorphic trivialization of 
$\XX(P)'|_N\to N$  that gets covered by a trivialization of $\Fcal'$ over $\XX(P)'|_N$. This defines the natural flat $\rho_P$-connection on $\Fcal$.

It is clear that the flat $\rho_{P}$-connection induced by $P_k$ differs by the one defined by $P$ by a differential that takes values in the complexification of
$R_u(\pfrak)+R_u(\pfrak_k)$. But this last space is contained in the nilpotent Lie algebra $R_u(\pfrak\cap\pfrak_1\cap\dots \cap \pfrak_n)$. 
\end{proof}

The quotient of $\Fcal$ by the $\G$-action gives a holomorphic vector bundle $\Fcal_\G$ on $\XX_\G$. The following corollary somewhat sharpens  the main result of \cite{gp}.

\begin{cordef}\label{cor:bbflatstructure}
The bundle $\Fcal_\G$ on $\XX_\G$ admits a natural flat $\rho_\G^\bb$-connection $\nabla_{\rho_\G^\bb}$. This structure is isoholonomic so that  we have defined the \emph{Goresky-Pardon Chern class lift} $\chern_k^\GP(\Fcal_\G):= \chern_k(\Fcal_\G,\nabla_{\rho_\G^\bb})\in H^{2k}(\XX^\bb_\G; \RR)$, $k=0,1,\dots$.
The restriction of such a class to the closure of a stratum is of the same type (it is the Chern class of an automorphic bundle on that stratum).
\end{cordef}
\begin{proof}
It is clear that Lemma \ref{lemma:bbflatness} produces a flat $\rho_\G^\bb$-connection  $\nabla_{\rho_\G^\bb}$ on $\Fcal_\G$. What remains to see is that
this structure is isoholonomic. The lemma in question identifies 
$\Fcal$ in a $P$-equivariant manner with the $\rho_P$-pull-back of a vector bundle $\Fcal(P)$ on $\XX(P)$. We may cover $\XX(P)$ by open subsets $V\subseteq \XX(P)$ such that $\Fcal (P)|V$ is trivial and the $\Gamma$-stabilizer of $V$
is in fact its $\G_P$-stabilizer. Such a $V$ then embeds as an open subset in a Baily-Borel stratum $\XX(P)_{\G (P)}$ of $\XX^\bb_\G$ and a trivialization of $\Fcal (P)|V$ yields the flat connection that is being asked for.
\end{proof}

\begin{remark}\label{rem:satext}
The argument used in the proof of Lemma \ref{lemma:bbflatness}  can probably be extended to prove that the automorphic bundles $\Fcal$ and $\{\Fcal (P)\}_{P\in\Pscr_{\max}}$ define a bundle $\Fcal^\bb$ over the Satake extension $\XX^\bb$  in the sense  that it becomes a locally free module over the structure sheaf of $\XX^\bb$. This has some interest, as it may somewhat simplify the proof in \cite{zucker:rbs2} of Conjecture 9.5 of \cite{gt}. But since $\G$  defines an arithmetic subgroup of $L_P^\ver$ which then may act nontrivially in the fibers of $\Fcal (P)$, $\Fcal^\bb_\G$  will in general not  be a locally free module over the structure sheaf of $\XX^\bb_\G$.
\end{remark}

\begin{example}\textbf{A symplectic group and its Hodge bundle.}\label{example:symplectic}
Let be given a finite dimensional real vector space $V$ endowed with a nondegenerate symplectic form $a: V\times V\to \RR$. 
The automorphism group of $(V,a)$ defines an almost simple algebraic group defined over $\RR$ whose group of real resp.\ complex points (endowed with the Hausdorff topology) is $\Sp (V)$ resp.\ $\Sp (V_\CC)$. Let $h: V_\CC\times V_\CC\to \CC$ be the Hermitian form defined by  $h(v,v'):=\sqrt{-1}a_\CC (v, \overline v')$. It has signature $(g,g)$. The associated symmetric space of $\Sp (V)$ and its compact dual are obtained as follows: $\check\XX=\check\HH (V)$ is the locus in the Grassmannian $\Gr_g(V_\CC)$ which parametrizes the $g$-dimensional subspaces $F\subseteq V_\CC$  that are totally isotropic relative to $a_\CC$ and $\XX=\HH(V)$ is the open subset of  $\check\HH (V)$ parametrizing those $F$ on which in addition $h$ is positive definite. The group $\Sp (V)$ indeed acts transitively on $\HH(V)$ and the stabilizer of any $[F]\in \HH(V)$ restricts isomorphically to the unitary group $U(F)$, which is a maximal compact subgroup of $\Sp (V)$. The restriction of the tautological rank $g$ bundle on  $\Gr_g(V_\CC)$ to $\HH(V)$ resp.\  $\check\HH (V)$ is an automorphic bundle $\Fcal=\Fcal_V$  resp.\ its natural extension $\check\Fcal$ to  $\check\HH (V)$. 
For $[F]\in \HH(V)$ we have $V_\CC=F\oplus\overline F$ and so $F$ is the $(1,0)$-part of a Hodge structure on $V$ of weight 1 polarized by $a$. Thus $\HH(V)$ also parametrizes polarized Hodge structures on $V$ of this type.  For this reason, $\Fcal$ is often called the \emph{Hodge bundle}. 
Notice that $h$ defines on $\Fcal$ an inner product (that we continue to denote by  $h$).

Now assume $V$ and $a$ are defined over $\QQ$ so that our  group is also defined over $\QQ$.
A maximal proper $\QQ$-parabolic subgroup $P\subseteq \Sp (V)$ is the $\Sp (V)$-stabilizer of a nonzero isotropic subspace $I\subseteq V$ defined over $\QQ$ and vice versa. So $\Pscr_{\max}$  may be identified with the set $\Iscr (V)$ of nonzero $\QQ$-isotropic subspaces of $V$. We will therefore index our  objects accordingly.  

Let $I\in \Iscr (V)$ and write $V'_I$ for $V/I$ and $V_I\subseteq V'_I$ for $I^\perp/I$. Note that  the symplectic form identifies $V'_I$ with the dual of $I^\perp$ and induces on $V_I$ a nondegenerate symplectic form.  Then the unipotent radical $R_u(P_I)$ of $P_I$  is the subgroup that acts trivially on $I$ and 
$V_I$; note that it then also acts trivially on $V/I^\perp$. This identifies the Levi quotient $L_I$ of $P_I$ with $\GL(I)\times \Sp (V_I)$. The center $U_I$ of $R_u(P_I)$ is the subgroup that acts trivially on $I^\perp$ (or equivalently, on $V'_I$). Its (abelian) Lie algebra $\ufrak_I$ can  be identified with  $\sym^2 I\subseteq \sym^2V\cong\gfrak$ and  $C_I\subseteq \sym^2 I$ is the cone of positive  definite elements. This identifies $R_u(P_I)/U_I$ with a group of elements in  $\GL(I^\perp)$ that act trivially on both $I$ and $V_I$; this group is abelian  with Lie algebra $\Hom (V_I, I)$ (which we shall identify with $I\otimes V_I$ by means  of the nondegenerate symplectic form on $V_I$). 
The central subgroup $A_I\subseteq P_I$ appears here as the group  of  scalars in $\GL(I)$, hence is a copy of $\RR^\times$.
The adjoint action of  $L_I=\GL(I)\times \Sp (V_I)$ on this Lie algebra is the obvious one.
In terms of these isomorphisms, $M_I^\hor=\{ \pm 1_I\}\times \Sp (V_I)$,  $M_I^\ver=\SL(I)$, $L_I^\ver=\GL(I)$,   $G_I=\Sp (V_I)$ and $P_I^\ver\subseteq G_I$ is the group that acts as $\pm 1$ on $V_I$.

We next describe the maps $\XX\to \XX(P_I)'\to \XX(P_I)$. 
For this we note that for $[F]\in \HH (V)$, the projection $F\to V'_{I,\CC}$ is into and the projection $F\to V_\CC/I^\perp_\CC\cong I^*_\CC$ is onto with kernel $F\cap I^\perp_\CC$ projecting isomorphically onto a subspace of $V_{I,\CC}$ that defines an element of $\HH(V_I)$. If we denote by $\HH(V'_I)$ the subspace of  Grassmannian of $V'_{I,\CC}$ parameterizing the subspaces whose intersection with $V_{I,\CC}$ defines an element of  $\HH(V_I)$, then we obtain a diagram
\[
\begin{CD}
\rho_I: \XX=\HH(V)@>{\rho'_I}>> \XX(P_I)'=\HH(V'_I)@>{\rho''_I}>>\XX(P_I)=\HH(V_I), 
\end{CD}
\]
where the maps are the obvious ones. It is clear that $\Fcal_V$ is the $\rho'_I$-pull-back of the tautological bundle $\Fcal_{V'_I}$ over $\HH(V'_I)$.
We note that $\HH(V'_I)\to \HH(V_I)$ is a principal bundle for the vector group  $\Hom (V/I^\perp,V_I)\cong V_I\otimes I$ and that this  $V_I\otimes I$-action lifts to $\Fcal_{V'_I}$. 

Let us determine the flat $\rho''_I$-connection on $\Fcal_{V'_I}$.  
First observe that any subspace $F''\subseteq V_I$ that defines an element of $\HH(V_I)$ determines a complex structure on 
$V_I$ characterized by the property that the $\RR$-linear isomorphism 
$V_I\subseteq V_{I,\CC}\to  V_{I,\CC}/F''\cong \Hom_\CC(F'',\CC)$ is in fact  $\CC$-linear. This gives the constant vector bundle on 
$\HH(V_I)$ with fiber $V_I$ a holomorphic structure (which can be identified with the total space of the dual $\Fcal_{V_I}^\vee$ of 
$\Fcal_{V_I}$). Hence the constant vector bundle on $\HH(V_I)$ with fiber $V_I\otimes I$ also acquires a holomorphic structure 
(namely as the total space of $\Fcal_{V_I}^\vee\otimes I$); let us denote that total space by $\VV_I$. Then 
$\rho''_I:\HH(V')\to \HH(V_I)$ is a $\VV_I$-torsor in the complex-analytic category.  A local section of $\rho''_I$ identifies 
$\Fcal_{V'_I}$ with the pull-back along  $\rho''_I$ of $\Fcal_{V_I}\oplus (\Ocal_{V_I}\otimes I^*)$, with the group 
$L_I=\Sp (V_I)\times \GL (I)$ acting in  the obvious way. This gives the flat $\rho''_I$-connection on $\Fcal_{V'_I}$.

The preceding also makes it clear that $P_I\le P_J$ is equivalent  to $I\subseteq J$. In other words, $(\Pscr_{\max}, \le)$ is 
identified with $(\Iscr(V),\subseteq)$. Note that for an inclusion $I\subseteq J$ of $\QQ$-isotropic subspaces, the diagram involving 
the associated cones is
\[
\begin{CD}
\sym^2I\subseteq \sym^2J@>>> \sym^2(J/I) \\
\cup &&\cup\\
C_I\le C_J@>>> C_{J/I}.
\end{CD}
\] 
\end{example}

\section{Resolving a Baily-Borel compactification as a stratified space}

In this section we describe the data needed for Mumford's toroidial compactifications introduced in \cite{amrt} and explain how this compares with the Baily-Borel construction. We will then show:
 
\begin{theorem}\label{thm:main2}
Every toroidal resolution  $\pi:\tilde\XX\to \XX^\bb_\G$ of the Baily-Borel compactification has the property that for every  Baily-Borel stratum, the holonomy of $\Fcal_\G$ relative to its local retraction is trivial over the preimage of that stratum in $\tilde\XX$ so that (by Proposition \ref{prop:main1}) $\Fcal_\G$ extends naturally to a holomorphic  vector bundle $\tilde\Fcal_\G$ on $\tilde X$ and we have $\pi^*\chern_k^\GP(\Fcal_\G)=\chern_k(\tilde\Fcal)_\CC$. We can  choose such a resolution to be part of a stratified resolution of $\XX^\bb_\G$ so that (by Theorem \ref{thm:main}) $\chern_k^\GP(\Fcal_\G)\in F^kH^k(\XX^\bb_\G; \CC)$.
\end{theorem}

For what follows it is, as a matter of notation, convenient to pretend that $G$ is also a maximal $\QQ$-parabolic subgroup:  since $R_u(G)$ is trivial we have  $C_G=\ufrak_G= \vfrak_G=0$ and  we take $G^\ver=\{1\}$,  $\XX(G)=\XX$.
 We write $\Pscr^*_{\max}=\Pscr_{\max}\cup\{ G\}$ for the corresponding collection. A partial order $\le$ on  $\Pscr^*_{\max}$ is defined as before and has $G$ as its minimal element.

The cones $\{C_P\}_{P\in\Pscr^*_{\max}}$ are pairwise disjoint as subsets of $\gfrak$. We denote their  union by $C(\gfrak)$ and write
$C_{P}^+$ for the union of  all the $C_Q$ with $Q\le P$ with $Q\in\Pscr^*_{\max}$ (so $0\in C_{P}^+$). Then $C_{P}^+$ is the closure of $C_P$ in $C(\gfrak)$ and is spanned by $\overline C_P\cap \vfrak_P(\QQ)$.
Now $\G_P:=\G\cap P$ is an arithmetic subgroup of $P$. In particular, $\G\cap R_u(P)$ is  an extension of a lattice 
(namely the image of $\G_P$ in $V_P$) by a lattice  (namely $\G\cap U_P$). The image of $\G_P$ in $\GL (\ufrak_P)$ preserves the  lattice
$\log (\G\cap U_P)$ in $\ufrak_P$ and therefore acts properly discretely  on $C_P$. 
An important feature of this action is that it has in $C_P^+$ a fundamental domain that is a \emph{rational polyhedral cone} 
(i.e., the convex cone spanned by a finite subset of $\log (\G\cap U_P)$). 
The extra ingredient needed for a toroidial compactification  is a \emph{$\G$-admissible decomposition} of $C(\gfrak)$, that is, a 
$\G$-invariant collection $\Sigma$ of rational polyhedral cones that is closed under `taking faces' and `taking intersections', 
whose relative interiors are pairwise disjoint and whose union is  $C(\gfrak)$. It is a basic fact \cite{amrt} that 
$\G$-admissible decompositions exist and that any two such have a common refinement. 

Given such a $\Sigma$, then the restriction of $\Sigma$ to the \emph{open} cone $C_P$, $\Sigma| C_P$, defines a relative 
torus embedding  $\XX_{\G\cap U_P}\subseteq\XX_{\G\cap U_P}^{\Sigma|C}$ over $\XX(P)'$. (Strictly speaking, there is not really a torus 
acting but rather an open semigroup in a torus, namely the image of  $\ufrak+\sqrt{-1}C_P$ in $U_P(\CC)/(\G\cap U_P)$ under the exponential map.) The result is a normal analytic variety with an action of the semigroup $\ufrak+\sqrt{-1}C_P$ and which has toroidal singularities. The group $\G_P/(\G_P\cap U_P)$ acts on it properly discontinuously. When we subsequently divide out by the image $\G(V_P)$ of $\G\cap R_u(P)$ in $V_P$ (which is just a lattice) we get a family of toroidal embeddings over $\XX(P)'_{\G(V_P)}$ where  the latter is now the total space  of a family of  abelian varieties
$\XX(P)'_{\G(V_P)}\to \XX(P)$ (or rather a torsor thereof). If we divide out by $\G_P/(\G_P\cap U_P)$ instead we get an  abelian torsor with base the 
Baily-Borel stratum $\XX(P)_{\G(G_P)}$ and $\XX_{\G_P}\subseteq\XX^{\Sigma|C}_{\G_P}$ appears as a toroidal embedding over the total space of this torsor.

It is perhaps more transparent, and also more in the Satake-Baily-Borel spirit, to do this construction before dividing out by 
$\G_P\cap U_P$, that is, to first introduce a $\G$-equivariant  extension $\XX^\Sigma$ of $\XX$ of ringed spaces. This brings us, 
like the Satake extension, outside the realm of analytic spaces, but the advantage of  this approach is that it allows us to  concisely 
describe the maps that exist between various compactifications. Here is how to proceed. For every $\sigma\in\Sigma$ we can form 
a holomorphic  quotient   $\rho_\sigma: \XX\to \XX(\sigma)$.  This map can be understood as the inclusion of $\XX$ in its  
$\exp (\la \sigma\ra_\CC)$-orbit in  the compact dual $\check{\XX}$ of $\XX$ ($\XX$ is an open subset of $\check{\XX}$) followed by 
the formation of the  $\exp (\la \sigma\ra_\CC)$-orbit space. Alternatively, $\rho_\sigma$ is the formation of the  quotient  of 
$\XX$ with respect to the equivalence relation generated by the relation 
$z\sim z'\Leftrightarrow z'\in \exp (\la \sigma\ra_\RR+\sqrt{-1}\sigma)z$. We let $\XX^\Sigma$  be the disjoint union of  the 
$\XX(\sigma)$'s (this includes $\XX=\XX (\{0\})$) and equip this union with the topology generated by the open subsets of 
$\XX$ and those of the form $\Omega^{\bb, \sigma}$, where 
$\sigma\in\Sigma$, $\Omega\subseteq \XX$ is an open subset  invariant under the semigroup $\exp (\la \sigma\ra_\RR+\sqrt{-1}\sigma)$, and
\[
\Omega^{\bb, \sigma}:=\bigsqcup_{\tau\in\Sigma; \tau\subseteq \overline\sigma} \rho_\tau(\Omega)
\]
(note that $\Omega$ appears in this union for $\tau=\{0\}$).   The structure sheaf is the sheaf of complex valued continuous functions that are holomorphic on each stratum.
Note that when $\sigma\subseteq C_P$, the map $\rho'_P:\XX\to \XX(P)'$ factors through  $\XX(\sigma)$. It is then clear that the composite projections 
$\XX(\sigma)\to \XX(P)'\to \XX(P)$ combine to define a continuous $\G$-equivariant  morphism $\pi^\Sigma: \XX^\Sigma\to \XX^\bb$ of locally ringed spaces, whose restriction over $\XX(P)$ in fact factors over 
$\XX(P)'$.  This drops to a morphism $\pi^\Sigma_\G: \XX^\Sigma_\G\to \XX^\bb_\G$ in the analytic category which has the property that  it factors over  a Baily-Borel stratum  through the abelian torsor that lies over it. 
We can now prove part of Theorem \ref{thm:main2}.

\begin{lemma}\label{lemma:toricversusbb}
The retractions $\XX\to \XX(\sigma)$ turn $\XX^\Sigma_\G$ into a rigidified stratified space such that $\pi^\Sigma_\G: \XX^\Sigma_\G\to \XX^\bb_\G$ is a morphism in this category. An automorphic bundle $\Fcal_\G$ on $\XX_\G$  satisfies the hypotheses of the last clause of Proposition \ref{prop:main1}  with respect to $\XX^\Sigma_\G$ and so the  total Chern class of the resulting extension $\Fcal^\Sigma_\G$ to $\XX^\Sigma_\G$ (with complex coefficients)  equals 
$(\pi^\Sigma_\G)^*\chern^{\GP}(\Fcal_\G)$.
\end{lemma}
\begin{proof}
Let $\sigma\in \Sigma$ be such that its relative interior is contained in $C_P$. Then $\rho_P$ (whose fibers are orbits of $R_u(P)\exp(\sqrt{-1}C_P)$)  factors through $\rho_\sigma$ (whose fibers are orbits $\exp (\la \sigma\ra_\RR+\sqrt{-1}\sigma)$). This proves the first assertion. The resulting local flat connections 
$\nabla_{\rho_\sigma}$ on our automorphic bundle are compatible: if $\tau\in\Sigma$ is a face of $\sigma$, then 
$\exp (\la \tau\ra_\RR+\sqrt{-1}\tau)\subseteq \exp (\la \sigma\ra_\RR+\sqrt{-1}\sigma)$ and so the local flat connection associated to $\tau$ induces the one associated to $\sigma$.
\end{proof}

\begin{remark}\label{rem:torext}
The extension $\Fcal^\Sigma_\G$ of $\Fcal_\G$  across $\XX^\Sigma_\G$ appears at various places in the literature; it is  the canonical extension described in \cite{mumford:hirz}. When $\Fcal$ belongs to the Hodge filtration of a locally homogeneous variation of Hodge structure, then it is also the Deligne extension.
Had we introduced the locally free $\Ocal_{\XX^\bb}$-module $\Fcal^\bb$ as in Remark \ref{rem:satext}, then we could say that $\pi^{\Sigma *}\Fcal^\bb$ is a locally free $\Ocal_{\XX^\Sigma}$-module with $\G$-action and $\Fcal^\Sigma_\G$ would simply be its $\G$-quotient (the $\G$-stabilizer of every $x\in \XX^\Sigma $ acts trivially on the fiber $\Fcal^\bb(x)$). 
\end{remark}

Let us  say that  the $\G$-admissible decomposition $\Sigma$ is \emph{smooth} if each member is an integral  simplicial cone  
(i.e., the cone spanned by an integral partial basis of $\log(\G\cap U_P)$ for some $P$). This ensures that $\XX^\Sigma_\G$ is 
smooth. In that case we will refer to $\XX^\Sigma_\G$ simply as a \emph{toroidal resolution} of  $\XX^\bb_\G$. 
Another basic fact is that any  $\G$-admissible decomposition admits a smooth refinement. 
The  following proposition will complete the proof Theorem \ref{thm:main2}.

\begin{theorem}\label{thm:bbres}
The Baily-Borel compactification $\XX^\bb_\G$  admits a toroidal resolution relative to its natural stratification in the sense of Definition 
\ref{def:stratares}.
\end{theorem}

\begin{proof}
In what follows we tacitly assume that the partitions $\Sigma$ of $C(\gfrak)$ we consider are so fine that for any $\sigma\in \Sigma$, the collection of $P\in \Pscr_{\max}^*$ for which $\sigma$ meets $C^+_P$  is a well-ordered subset of $\Pscr_{\max}^*$. The same applies to partitions of  the cones $C(\gfrak_P)$.

The first question we must address is the following. Let $P\in\Pscr_{\max}$ and suppose we are given a $\G$-admissible decomposition $\Sigma$ of $C(\gfrak)$ and a $\G(G_P)$-admissible decomposition $\Sigma(P)$ of $C(\gfrak_P)$. The former defines 
$\pi^\Sigma:\XX^\Sigma\to \XX^\bb$ and the latter defines $\pi^{\Sigma(P)}:\XX^{\Sigma(P)}\to \XX(P)^\bb$ and we want to know  when  the restriction of $\pi^\Sigma$ to the closure of $(\pi^\Sigma)^{-1}\XX(P)$ in $\XX^\Sigma$ factors through $\pi^{\Sigma(P)}$. For $(\pi^\Sigma)^{-1}\XX(P)$ itself there is no issue: we have a factorization  $(\pi^\Sigma)^{-1}\XX(P)\to \XX^{\Sigma(P)}\to \XX(P)^\bb$.

A rational boundary component of  $\XX(P)^\bb\ssm \XX(P)$  is of the form $\XX(Q)$, with  $Q\in\Pscr_{\max}$ such  that $Q>P$, or equivalently, $C_P\subset C_Q^+$. 
A  stratum of $\XX^{\Sigma}$ over $\XX(Q)$ that lies in the closure of a stratum over $\XX(P)$ is of the form $\XX(\sigma)$, 
with $\sigma\in\Sigma |C^+_Q$  such that $\sigma$ meets  $C_Q$ and  $(\sigma\ssm \{ 0\}) \cap C^+_P$  is nonempty and contained in $C_P$ (recall that $\XX(\sigma)$ is the quotient of $\XX$ by the equivalence relation generated by $z\sim z'\Leftrightarrow z'\in \exp (\la \sigma\ra_\RR+\sqrt{-1}\sigma)$). On the other hand,  a  stratum of $\XX(P)^{\Sigma (P)}$ over 
$\XX(Q)$ is of the form $\XX(P)(\tau)$, where  $\tau\in\Sigma (P)$ is such that the relative interior of $\tau$ lies in $C_{Q_{/P}}$. We obtain it as a quotient of the equivalence relation on $\XX(P)$ generated by $z\sim z'\Leftrightarrow z'\in \exp (\la \tau\ra_\RR+\sqrt{-1}\tau)$. Let us now  also recall that $C_{Q_{/P}}$ is the image of $C_Q$ under the projection  $\ufrak_Q\to \ufrak_Q/ \ufrak_Q\cap R_u(\pfrak)\cong \ufrak_{Q_{/P}}$. So $\XX(\sigma)$ maps onto $\XX(\tau)$ if and only if this projection maps $\sigma$ to the relative interior of $\tau$. In other words, we want 
that this projection maps any member of $\Sigma$ in the star of $C_P$ in $C_Q^+$ to a member of $\Sigma (P)$.

This reduces the proposition to a combinatorial issue: we must construct for
every  $P\in\Pscr^*_{\max}$ a  $\G(G_P)$-admissible decomposition $\Sigma(P)$ of $C(\gfrak_P)$ such that 
\begin{enumerate}
\item [(i)] $\g\in\Gamma$ takes $\Sigma(P)$ to $\Sigma(\gamma P\gamma^{-1})$,
\item [(ii)] for every chain of triples $Q\ge P\ge P_0$ in $\Pscr_{\max}^*$, the projection  
\[
\ufrak_{Q_{/P_0}}\cong \ufrak_Q/ \ufrak_Q\cap R_u(\pfrak_0) \to \ufrak_Q/ \ufrak_Q\cap R_u(\pfrak)\cong \ufrak_{Q_{/P}}
\]
maps every member of $\Sigma (P_0)$ in the star of $C_{P_{/P_0}}$ to a member of $\Sigma (P)$.
\end{enumerate}

We begin  with choosing a $\Sigma (Q)$ for every member $Q$ of $\Pscr^*_{\max}$ that is maximal for $\le$ such that 
$(i)$ is satisfied. We then proceed with downward induction on the partially ordered set $(\Pscr^*_{\max}, \le )$ and assume that 
we have constructed  for every $P\in \Pscr_{\max}$ a $\G(G_P)$-admissible decomposition $\Sigma(P)$ of $C(\gfrak_P)$ 
satisfying (i) and (ii), so that it remains to construct $\Sigma=\Sigma(G)$. 

For every maximal element $P$ of  $\Pscr_{\max}$ we choose a rationally polyhedral cone $\Pi_P\subseteq C_P^+$ that is a fundamental 
domain for the action of $\G_P$ on $C_P^+$ in such a manner
 that $\Pi_{\g P}=\g (\Pi_P)$.  For every face $Q\le P$ such that $\Pi_P\cap C_Q\not=\emptyset$ the image of $\Pi_P$ in $C^+_{P_{/Q}}$ is a 
 rationally polyhedral cone and so meets  only a finite number of members of $\Sigma (Q)$. Hence the pull-back of $\Sigma (Q)$ to $\Pi_P$ is 
 a finite decomposition of $\Pi_P$ into rationally polyhedral cones. The set of $Q$ with $\Pi\cap C_Q\not=\emptyset$ is also finite and so the 
 finitely intersections of these pull-backs make up a decomposition $\Sigma (\Pi_P)$ of $\Pi_P$ into finitely many rationally polyhedral cones. 

Now let $P$ run over a system of representatives $\{P_i\}_{i=1}^r$ of the $\G$-action in the collection of maximal elements of $\Pscr_{\max}$. 
So for each $i$ we have a rationally polyhedral cone $\Pi_i$ and a decomposition $\Sigma (\Pi_{P_i})$ of that cone. Choose a  
$\G$-invariant admissible decomposition  $\Sigma$ which refines each $\Sigma (\Pi_{P_i})$.  After possibly refining once more we can 
arrange $\Sigma$ that be smooth. It will then have the desired properties.
\end{proof}

\section{Tate extensions in the stable cohomology of $\Acal_g^\bb$}

\subsection*{The  stable cohomology of $\Acal_g^\bb$} We  here focus on what is perhaps the most `classical' example and  
also is a special case of \ref{example:symplectic},  namely the moduli stack  $\AA_g$ of principally polarized abelian varieties. 
We shall prove that the stable cohomology of its Baily-Borel compactification contains nontrivial Tate extensions and carries 
Goresky-Pardon Chern classes that have nonzero imaginary part (and hence are not defined over $\QQ$).

Let $H$ stand for  $\ZZ^2$ and endowed with standard symplectic form (characterized by  $\la e, e'\ra=1$ where $(e,e')$ is its standard basis) 
and  regard $H^g(=\ZZ^{2g})$  as a direct sum of symplectic lattices. In the notation of Example \ref{example:symplectic} we take for $V$ the 
vector space $\RR\otimes H^g(=\RR^{2g})$ with its obvious rational symplectic structure so that we have defined the symmetric domain 
$\HH_g:=\HH (\RR\otimes H^g)$ and 
we take for $\G$ the integral symplectic group $\Sp (H^g)(=\Sp (2g, \ZZ))$. Then $\AA_g$ can be identified with  $\Sp (H^g)\bs \HH_g$, 
when we think of the latter as a Deligne-Mumford stack. The Hodge bundle on $\HH_g$ descends to a rank $g$ vector bundle $\Fcal_g$ on the 
stack $\Sp (H^g)\bs \HH_g$. As such it has integral Chern classes. In what follows we will work mostly with cohomology with coefficients in 
$\QQ$-vector spaces. 
Then the distinction between the stack $\AA_g$ and underlying coarse moduli space (that we shall denote by $\Acal_g$) becomes moot, for the natural map 
from $\AA_g$ (which has the homotopy type of $B\Sp(2g, \ZZ)$)  to $\Acal_g$ induces an isomorphism on rational (co)homology. The Hodge bundle on $\HH_g$ descends to a bundle $\Fcal_g$ on 
$\AA_g$ and thus we find $\cha_k(\Fcal_g)\in H^{2k}(\AA_g; \QQ)\cong H^{2k}(\Acal_g; \QQ)$. We will therefore pretend that $\Fcal_g$ is a vector bundle on $\Acal_g$.
According to Charney and Lee \cite{charney-lee}, $H^k(\Acal_g^\bb; \QQ)$ is  independent of $k$ for $g$ sufficiently large. 
They prove that the direct sum of these stable cohomology spaces comes with the structure of a connected $\QQ$-Hopf algebra  $H^\pt$ whose primitive generators are classes $\widetilde{ch}_{2r+1}\in H^{4r+2}$  ($r\ge 0$) and classes $y_r\in H^{4r+2}$ ($r\ge 1$). For $g\gg r$, the image of 
$\widetilde{ch}_{2r+1}\in H^{4r+2}$ in $H^{4r+2}(\Acal_g; \QQ)$ is $\cha_{2r+1}(\Fcal_g)$ (which is known to be nonzero), whereas the image of $y_r$ in $H^{4r+2}(\Acal_g; \QQ)$ is zero.

The class $y_r$ is somewhat harder to describe: it comes from transgression of a primitive class in $H^{4r+1}(B\GL (\ZZ); \QQ)$ about which we will say more below.  Jiaming Chen and the  author \cite{chen-looijenga}  have recently shown that the stability theorem holds if we take the mixed Hodge structure  on $H^\pt(\Acal_g^\bb; \QQ)$ into account: $H^\pt$ inherits such a structure with $H^k$ having weight  $\le k$ (for $\Acal_g^\bb$ is compact) and 
$H^k/W_{k-1}H^k$ can be identified with  $H^k(\Acal_g; \QQ)$ for $g$ large. For $k=4r+2$  the image of $\widetilde{ch}_{2r+1}$ in  
$H^{4r+2}(\Acal_g; \QQ)$ 
 is $\cha_{2r+1}(\Fcal_g)$, which is of  bidegree $(2r+1,2r+1)$ (and nonzero for $g\gg r$), but  $y_r$ ($r\ge 1$) is of  bidegree $(0,0)$.
So the primitive part $H^{4r+2}_{\prim}$ of $H^{4r+2}$ is for $r\ge 1$ a Tate extension:
\begin{equation}\label{eq:4}
0\to \QQ(0)\to H^{4r+2}_{\prim} \to \QQ(-2r-1)\to 0,
\end{equation}
where $\QQ(-2r-1)$ is spanned by the image $ch_{2r+1}$ of  $\widetilde{ch}_{2r+1}$ and $\QQ(0)$ by the image of $y_r$. 
The inclusion $\QQ(-2r-1)\subseteq \CC$ comes about by regarding the  twisted (De Rham) version $(2\pi\sqrt{-1})^{2r+1}{ch}_{2r+1}$ as the natural generator (it lies in $\QQ(0)$).

In what follows we take $g$ large enough to be in the stable range, so that this sequence appears in $H^{4r+2}(\Acal_g^\bb, \QQ)$. By Theorem \ref{thm:main2} (in combination with  Remarks \ref{rem:chat} and \ref{rem:chhodge}), the Goresky-Pardon Chern character $\cha_{2r+1}^{\GP}(\Fcal_g)$ (being a universal polynomial with rational coefficients of weighted degree $2r+1$ in the $\chern_i^{\GP}(\Fcal_g)$) is then a generator of $F^{2r+1}H^{4r+2}_{\prim}$. This will help us determine the class of this extension. For this purpose  we also need to know a bit more about $y_r$, when viewed as an element of $H^{4r+2}(\Acal_g^\bb; \QQ)$. We will however not describe $y_r$, but rather a stable primitive \emph{homology} class $z_r\in H_{4r+2}(\Acal_ g^\bb; \QQ)$ such that $\la y_r,z_r\ra\not= 0$. That will do, for then the map  $x\in H^{4r+2}_{\prim}\mapsto \la x,z_r\ra/\la y_r,z_r\ra \in \QQ=\QQ(0)$ splits the  above sequence and so the extension class is given by the image of $\la \cha_{2r+1}^{\GP}(\Fcal_g), z_r\ra$ in $\CC/\QQ$. We prefer to replace 
$\cha_{2r+1}^{\GP}(\Fcal_g)$ by its De Rham variant $(2\pi\sqrt{-1})^{2r+1}\cha_{2r+1}^{\GP}(\Fcal_g)$, so that the class of  this Tate extension becomes more like a period; it is then the image of  $\la (2\pi\sqrt{-1})^{2r+1}\cha_{2r+1}^{\GP}(\Fcal_g), z_r\ra$ in $\CC/\QQ(2r+1)$.
The following theorem implies that this extension is nontrivial and  that the Goresky-Pardon Chern character has a nonzero imaginary part. 

\begin{theorem}\label{thm:main3}
The class of the Tate extension (\ref{eq:4}) in $\CC/\QQ(2r+1)$ (which is given by the image of   $\la (2\pi\sqrt{-1})^{2r+1}\cha_{2r+1}^{\GP}(\Fcal_g), z_r\ra$ in $\CC/\QQ(2r+1)$) is real  and equal to a nonzero rational multiple of $\pi^{-2r-1}\zeta (2r+1)$. In particular,  the imaginary part of $\cha_{2r+1}^{\GP}(\Fcal_g)$ is nonzero and its real part lies in  $H^{4r+2}(\Acal_g^\bb;\QQ)$.
\end{theorem} 

The computation  uses Beilinson's regulator for the field $\QQ$, which involves among other things Deligne cohomology and the Cheeger-Simons classes. We recall what we need below, referring to  Burgos' very accessible exposition \cite{burgos} as  a general reference for this topic.  
\subsection*{Refined Chern characters}
For a smooth complex variety $X$ there is defined the \emph{Deligne cohomology group} $H^{2p}_{\Dcal} (X, \ZZ(p))$ ($p=0,1,2, \dots$). It fits in an exact sequence
\[
0\to J_p(X)\to H^{2p}_{\Dcal} (X; \ZZ(p))\to F^pH^{2p}(X; \ZZ(p))\to 0,
\]
where  $F^pH^{2p}(X, \ZZ(p))$ denotes the intersection of the image of $H^{2p}(X;\ZZ(p))\to H^{2p}(X; \CC)$ with $ F^pH^{2p}(X; \CC)$ and $J_p(X)$ is an abelian group that is the $p$th  intermediate Jacobian in case $X$ is projective: 
\[
J_p(X):=H^{2p-1}(X; \CC)/\big(F^pH^{2p-1}(X; \CC)+H^{2p-1}(X; \ZZ(p))\big).
\]
We only need here the following somewhat informal description of this extension: when $X$ is complete, an element of $H^{2p}_{\Dcal} (X; \ZZ(p))$ is representable by a pair $(b, \alpha)$, where $b\in H^{2p-1}(X;\CC/\ZZ(p))$ and $\alpha$ is closed $2p$-form on $X$ of Hodge level $\ge p$ with periods in $\ZZ(p)$ (we then write $\alpha \in (F^pA)^{2p}_{\text{cl}}(X; \ZZ(p)$), such that for every smooth singular $\ZZ$-valued $2p$-chain $Z$ on $X$, the image of  $\int_Z\alpha$ in $\CC/\ZZ(p)$ is equal to $b([\partial Z])$. In case $X$ is not complete, we  require that $\alpha$  extends to a normal crossing compactification with logarithmic poles along $D$ of $X$ (so that it represents an element of $F^pH^{2p}(X)$ with periods in $\ZZ(p)$). The equivalence relation is the one which produces the exact sequence and so $(b, \alpha)$ represents zero precisely when the cohomology class of $\alpha$ is zero and  $b$ is in the image of $F^pH^{2p-1}(X; \CC)\to H^{2p-1}(X; \CC)/H^{2p-1}(X; \ZZ(p))=H^{2p-1}(X; \CC/\ZZ(p))$. 
Beilinson and Gillet showed that for a vector bundle $\Fcal$ on $X$ one has a natural lift of $(2\pi\sqrt{-1})^p \cha_{2p}(\Fcal)\in F^pH^{2p}(X; \ZZ(p))$ to $H^{2p}_{\Dcal} (X; \ZZ(p))$.  It is called the  \emph{Beilinson Chern character} and---in order to come to terms with the fact that  Beilinson  and Betti have a common initial string---we denote it  by  $\cB_{p}(\Fcal)\in H^{2p}_{\Dcal} (X; \ZZ(p))$.

It was observed by  Dupont, Hain and Zucker \cite{dhz} that we can also get this class as a Cheeger-Simons  differential character, which is defined in a $C^\infty$-setting.
For a manifold $M$ we have an extension that is similarly defined as $H^{2p}_{\Dcal} (X; \ZZ(p))$ above:
\[
0\to H^{2p-1}(M; \CC/\ZZ(p))\to\hat H^{2p}(M; \CC/\ZZ(p))\to A_{\text{cl}}^{2p}(M; \ZZ(p))\to 0,
\]
where $A_{\text{cl}}^{2p}(M; \ZZ(p))$ denotes the space of closed $2p$-forms on $M$ with periods in $\ZZ(p)$. A complex vector bundle $\Fcal$ on $M$ endowed  with a connection $\nabla$ defines \emph{Cheeger-Simons Chern character} $\widehat\cha_p (\Fcal, \nabla)\in  \hat H^{2p}(M;\CC/\ZZ(p))$, the closed $2p$-form $\Ch_p(\Fcal, \nabla)$  then being given as $\Tr ((-R(\nabla))^p)/p!$, where $R(\nabla)\in A_{\text{cl}}^{2}(\End (\Fcal))$ denotes the curvature form of $\nabla$.
Dupont-Hain-Zucker \cite{dhz} verified the compatibility with the Beilinson's Chern character: if $X$ is projective and $\Fcal$ is an algebraic vector bundle  endowed with a connection
$\nabla$ of type  $(1,0)$, then the Chern character form $\Ch_p(\Fcal, \nabla)$ lands in $(F^pA)^{2p}_{\text{cl}}(M; \ZZ(p))$. This ensures that
$\widehat\cha_p(\Fcal, \nabla)$  maps to the corresponding subspace $F^p\hat H^{2p}(M; \CC/\ZZ(p))$ of $\hat H^{2p}(M; \CC/\ZZ(p))$ and the evident 
projection $F^p\hat H^{2p}(M; \CC/\ZZ)\to H^{2p}_{\Dcal} (X; \ZZ(p))$ maps $\widehat\cha_p(\Fcal; \nabla)$ to $\cB_p(\Fcal)$. This is then also true when $X$ is quasi-projective, provided we know that $(\Fcal, \nabla)$ extends across a   
smooth normal crossing compactification, for both  refinements of the  Chern character behave functorially with respect to pull-backs. 

\subsection*{The regulator map for $\QQ$}
The  group homology of $\GL(g,\ZZ)$ stabilizes in $g$ and the resulting stable rational homology is a graded commutative $\QQ$-Hopf algebra with a primitive generator for each degree $4r+1$ (so it is an exterior algebra). This stable homology is in fact the rational homology of  $B\GL(\ZZ)$, where $\GL(\ZZ)$ is the monotone union   
$\cdots\subseteq \GL (g,\ZZ)\subseteq \GL (g+1,\ZZ)\subseteq \cdots$.  Applying  Quillen's plus construction does not affect the homology and hence this remains so for the homology of $B\GL(\ZZ)^+$.
The latter is an $H$-space with distinguished generators up to sign for  its primitive rational homology: following  Quillen, the algebraic $K$-groups of $\ZZ$  are defined as $K_s(\ZZ):=\pi_s (B\GL(\ZZ)^+,*)$  and the Hurewicz map 
\[
K_\pt (\ZZ)=\pi_\pt(B\GL(\ZZ)^+,*)\to H_\pt(B\GL(\ZZ)^+)\cong H_\pt(B\GL(\ZZ))
\] 
induces for $s>0$ an isomorphism of $K_s (\ZZ)\otimes\QQ$ onto $H^\prim_s(B\GL(\ZZ);\QQ)$. It is known for that $s>0$, $K_s (\ZZ)$ is a torsion group unless $s=4r+1$ ($r=0,1,\dots$)  in which case it has rank one. We choose for  $r>0$ a generator $b_r$ of the image of  
$K_{4r+1}(\ZZ)\to K_{4r+1}(\ZZ)\otimes \QQ$ and identify it with its image in $H^\prim_{4r+1}(B\GL(\ZZ); \QQ)$. 
This element is of course defined up to sign. 
Over $B\GL(g, \ZZ)$ we have the universal local system $\VV_g$ with fiber $\ZZ^g$. The inclusion $\GL (g, \ZZ)\subseteq \GL (g, \CC)$ induces a map
$B\GL (g, \ZZ)\to B\GL (g, \CC)$.  If we take direct limits, then the
resulting map $B\GL(\ZZ)\to B\GL(\CC)$ is zero on rational homology in positive degree (being a homomorphism from an exterior algebra to a polynomial one), but the situation is different for 
Deligne cohomology. This of course requires that we are in an algebraic setting, which is kind of clear for $B\GL(\CC)$, being an inductive limit of
Grassmannians, but less so for $B\GL(\ZZ)$. Yet, as explained in \cite{dhz} and \cite{burgos}, this can be given a sense by regarding $B\GL(\ZZ)$ as a simplicial projective manifold of dimension zero (and in order to get the map, we must then do the same for $B\GL(\CC)$).\footnote{This can probably also be used to produce another proof that $y_r$ is of type $(0,0)$.} 

We are interested in the value $\cB_{2r+1} (\VV_g) (b_r)\in \CC/\QQ(2r+1)$, or rather its image in $\CC/\RR(2r+1)$. Since $\RR(2r+1)$ is just the imaginary axis, we may  identify  $\CC/\RR(2r+1)$ with $\RR$ so that we have  a natural map $\CC/\QQ(2r+1)\to \CC/\RR(2r+1)\cong\RR $. The
image of  $\cB_{2r+1}(\VV_g) (b_r) \in \CC/\QQ(2r+1)$ in $\RR $ is according to Beilinson \cite{burgos} given  by a rational 
multiple of  the corresponding regulator of $\QQ$, which is $\zeta'(-2r)$, 
where $\zeta$ is the classical Riemann zeta function. (It is in fact known that  $\cB_{2r+1}(\VV_g) (b_r)$ itself  is represented
by $\zeta'(-2r)$, but we will obtain this as an outcome of our computation.) If we then invoke 
the functional equation for $\zeta$, we find:

\begin{scholium}\label{scholium}
The image of $\cB_{2r+1}(\VV_g)(b)\in \CC/\QQ(2r+1)$  under the natural map  $\CC/\QQ(2r+1)\to \RR $  is a nonzero rational multiple of   
$\pi^{-2r-1}\zeta(2r+1)$.
\end{scholium}

\subsection*{Proof of Theorem \ref{thm:main3}}
Returning to the situation at hand, let us denote by $I$ resp.\ $I'$  the integral span of the first basis resp.\ second  basis element of $H=\ZZ^2$, so that 
we have a decomposition $H^g=I^g\oplus I'{}^g$ into maximal isotropic sublattices of $H^g$. The symplectic form identifies $I'$ with $\Hom (I^g, \ZZ)$ and so we have an embedding $\GL(g, \ZZ)=\GL (I^g)\hookrightarrow  \Sp (H^g)$ defined by $\sigma\mapsto (\sigma, (\sigma^*)^{-1})$. This map commutes with the stability maps on either side so that the map on rational homology also stabilizes, but this will yield the zero map as $H_\pt(B\GL (\ZZ); \QQ)$ is an exterior algebra and $H_\pt(B\Sp (\ZZ); \QQ)$ a  polynomial algebra. However, as explained in \cite{chen-looijenga}, if $\infty\in \Acal_g^\bb$ is the worst cusp (the unique element of the zero-dimensional Satake stratum $\Acal_0$ of $\Acal_g^\bb$), then we have a basis of regular neighborhoods $U_\infty$ of $\infty$ in $\Acal_g^\bb$ with the property that  $\mathring{U}_\infty:= U_\infty\cap \Acal_g$ is a virtual classifying space for
the semi-direct product $\GL(g,\ZZ)\ltimes \sym^2 (\ZZ^g)$ and so  contains a virtual classifying space for  $\GL(g,\ZZ)$. We will make use of the fact that this virtual classifying space can be chosen in the real locus. Here  we note that the modular interpretation of $(\Acal_g, \Fcal_g)$ endows this pair with a real structure.
The Baily-Borel compactification $\Acal_g^\bb$ together with its stratification are defined over $\RR$. In particular, $\infty$ is a real point so that we can take $U_\infty$ invariant under complex conjugation.

\begin{lemma}\label{lemma:real}
The locus $U_\infty\cap \Acal_g(\RR)$ is a virtual classifying space for $\GL(g, \ZZ)$ and so we can represent $b_r$ by a cycle  $B_r$  on 
$U_\infty\cap \Acal_g(\RR)$. 
\end{lemma}
\begin{proof}
The real structure on $\Acal_g$ lifts to one on $\HH_g$, which  in relation to the cusp $\infty$ is best understood in terms of the Siegel upper half plane model. The symplectic form identifies the space of complex symmetric tensors $\sym^2(I^g_\CC)$ with the space of symmetric maps $I^g_\CC\to I'{}^g_\CC$. The graph of such a map lies in $\HH_g$ if and only if the imaginary part of the symmetric tensor is positive.  If $C_{I^g}$ denotes the  locus $C_{I^g}$  of  positive 
symmetric tensors, then $\sqrt{-1}C_{I^g}$ defines a real subset of $\HH_g$. 
The $\Sp (H^g)$-stabilizer of $\sqrt{-1}C_{I^g}$ is $\GL(I^g)$ and the orbit space $\GL(I^g)\bs \sqrt{-1}C_{I^g}$ maps onto a connected component of the real locus of $\Acal_g$.
Now $\GL(I^g)\bs C_{I^g}$ is a virtual classifying space for $\GL(I^g)=\GL(g, \ZZ)$. This is still so if  we replace $C_{I^g}$ by any $\GL(I^g)$-invariant \emph{cocore} $K\subseteq C_{I^g}$ \cite{amrt}. In particular,  $\GL(I^g)\bs(\sqrt{-1}K)$ supports a $(4r+1)$-cycle $B_r(K)$ which represents the primitive element $b_r$ defined above. For an appropriate choice of $K$,   $\GL(I^g)\bs(\sqrt{-1}K)$ embeds in $\mathring{U}_\infty$ and we then take $B_r$ to be the image of 
$B_r(K)$. 
\end{proof}

Since $H_{4r+1}(\Acal_g; \QQ)=0$ (we are in the stable range), the cycle $B_r$ bounds a $\QQ$-chain $Z_r$ in $\Acal_g$. As $U_\infty$ is contractible 
(even conical we make a careful choice for $U_\infty$), this cycle also bounds a chain $cB_r$ in $U_\infty\cap \Acal^\bb_g(\RR)$ so that we obtain a $(4r+2)$-cycle  $Z_r-cB_r$ on $\Acal_g^\bb$. It is shown  in  \cite{chen-looijenga} that the stable cohomology class $y_r\in H^{4r+2}(\Acal_g^\bb; \QQ)$ takes a 
nonzero value on this class so that $[Z_r-cB_r]$ may serve as our $z_r\in H_{4r+2}(\Acal_g^\bb; \QQ)$. It remains to compute the value of 
$(2\pi\sqrt{-1})^{2r+1}\cha_{2r+1}^\GP(\Fcal_g)$  on $[Z_r-cB_r]$.

Corollary  \ref{cor:bbflatstructure}  gives us a connection $\nabla$ on $\Fcal_g$ whose curvature form yields  the twisted Goresky-Pardon Chern characters. According  to Corollary \ref{cor:real} these are invariant under  full complex conjugation.  We assume that  $U_\infty$ has been chosen so small that $\nabla$ is flat on  $\mathring{U}_\infty$ and defines on $U_\infty\cap \Acal_g(\RR)$ a local system given by the obvious  representation of degree $g$ of $\GL(I^g)$. Then the form $\Ch_{2r+1}(\Fcal_g, \nabla)$ vanishes on $cB_r$ and so we find that
\[
\la (2\pi\sqrt{-1})^{2r+1}\cha_{2r+1}^\GP(\Fcal_g), z_r\ra =\int_{Z_r} \Ch_{2r+1}(\Fcal_g, \nabla)=\widehat\cha_{2r+1}(\Fcal_g, \nabla)(Z_r).
\]
As $\Ch_{2r+1}(\Fcal_g, \nabla)$ defines a class in  $H^{4r+2}(\Acal_g;\QQ(2r+1))$, the image of this integral in $\CC/ \QQ(2r+1)$ only depends on $\partial Z_r=B_r$ and is then given by  the value $\widehat\cha_{2r+1}(\Fcal_g)(b_r)\in \CC/ \QQ(2r+1)$. Since $\Ch_{2r+1}(\Fcal_g, \nabla)$ and $b_r$ are invariant under full complex conjugation, this value lies in fact in the image of $\RR$ in  $\CC/ \QQ(2r+1)$. In other words, it is completely given by its image 
in  $\CC/ \RR(2r+1)\cong\RR$. We have observed that $(\Fcal, \nabla)$ extends as a holomorphic vector bundle with flat connection to a nonsingular toric compactification and so this is also equal to $\cB_{2r+1}(\Fcal_g)(b_r)\in \CC/ \QQ(2r+1)$. According to our Scholium \ref{scholium} its image in $\CC/ \RR(2r+1)\cong\RR$ is a rational multiple of $\pi^{-2r-1}\zeta (2r+1)$. This completes the proof.
\\

\begin{cremarks}
Let us adhere to the custom to denote  $i$th Chern class of  the Hodge bundle on $\Acal_g$ by $\lambda_i\in H^{2i}(\Acal_g; \QQ)$. The Goresky-Pardon lift of $\lambda_i$ to $\Acal_g^\bb$ is in fact a De Rham lift $\lambda^\GP_i\in F^iH^{2i}(\Acal^\bb_g; \CC)$, which, as we have seen, sometimes not even lies in 
$H^{2i}(\Acal_g; \RR)$. However, for any toric resolution $\pi: \Acal_g^\Sigma\to \Acal^\bb_g$, the Hodge bundle on $\Acal_g$ extends canonically to $\Acal_g^\Sigma$ so that we do have a canonical lift $\lambda^\Sigma_i\in H^{2i}(\Acal^\Sigma_g; \QQ)$.  According to  Proposition \ref{prop:main1}, 
the image of $\lambda^\Sigma_i$ in  $H^{2i}(\Acal^\Sigma_g; \CC)$ equals $\pi^*\lambda^\GP_i$ and so applying $\pi^*$ drastically simplifies things (in particular, $\pi^*\lambda^\GP_i$ lies in  $H^{2i}(\Acal^\bb_g; \QQ)$).  If we are in the stable range ($2i<g$), then according to Charney-Lee,  $\lambda_i$ extends to a class $\tilde\lambda_i\in H^{2i}(\Acal^\bb_g; \QQ)$, but this lift is not unique. Yet its image under $\pi^*$ is still $\lambda^\Sigma_i\in H^{2i}(\Acal^\bb_g; \QQ)$. So the difference $\lambda^\GP_i- \tilde\lambda_i$ determines 
the nature of a Tate extension, and this extension  becomes trivial when pulled back to $H^{2i}(\Acal^\Sigma_g; \QQ)$.

Richard Hain computed  in \cite{hain} the rational cohomology for $\Acal_g$ and $\Acal_g^\bb$ (with their mixed Hodge structure) for $g=2, 3$.  He found that for $g=2$  all the rational cohomology is generated by  $\lambda_1$ (and so is not so interesting), but that $H^6(\Acal_3; \QQ)$ and $H^6(\Acal^\bb_3; \QQ)$ contain possibly nontrivial Tate extensions. For example, $H^6(\Acal_3; \QQ)$ is an extension of $\QQ(-6)$ by $\QQ(-3)$ and hence $H^6_c(\Acal_3; \QQ)\cong H^6(\Acal^\bb_3, \Acal^\bb_2; \QQ)$ (which embeds in $H^6(\Acal^\bb_3; \QQ)$ as a subspace of codimension one)
is an extension of $\QQ(-3)$ by $\QQ(0)$.  So this very much looks like the stable cohomology of $H^6(\Acal^\bb_g)$, although we are here of course outside the stable range (which requires $g>6$).
Hain  raises the question whether this extension is nontrivial and more specifically, whether it is of the type that we have been discussing here. Our results have nothing to say about this (although the techniques used here could  be helpful), but  at least they do suggest to investigate whether the following holds:  \emph{Is for $g>6$ the restriction map $H^6(\Acal^\bb_g)\to H^6(\Acal^\bb_3)$ an injection?} 
\end{cremarks}

\end{document}